\documentclass[10pt]{article}
\usepackage{amsmath,amssymb}
\usepackage{graphicx}
\usepackage{mathrsfs}
\usepackage{url}
\usepackage{xypic}\xyoption{tips}

\renewcommand{\AA}{\mathbb{A}}
\newcommand{\PP}{\mathbb{P}}
\newcommand{\FF}{\mathbb{F}}
\newcommand{\ZZ}{\mathbb{Z}}
\newcommand{\cF}{\mathscr{F}}
\newcommand{\cI}{\mathscr{I}}
\newcommand{\cK}{\mathscr{K}}
\renewcommand{\cL}{\mathscr{L}}
\newcommand{\cM}{\mathscr{M}}
\newcommand{\cO}{\mathcal{O}}
\newcommand{\oO}{O}

\newcommand{\eval}{ev} 
\newcommand{\isom}{\cong}
\newcommand{\End}{\mathrm{End}}

\newcommand{\Hom}{\mathrm{Hom}}
\newcommand{\Pic}{\mathrm{Pic}}
\newcommand{\NS}{\mathrm{NS}}
\newcommand{\Num}{\mathrm{Num}}
\newcommand{\ignore}[1]{}
\newcommand{\ccomma}{\raisebox{0.4ex}{,}}
\newcommand{\fs}{\mathfrak{s}}
\newcommand{\ft}{\mathfrak{t}}
\newcommand{\addition}[1]{#1}

\newtheorem{theorem}{Theorem}
\newtheorem{lemma}[theorem]{Lemma}
\newtheorem{corollary}[theorem]{Corollary}

\newenvironment{proof}{{\bf Proof}}{}

\newcommand{\keyword}[1]{}
\newcommand{\address}[1]{}
\newcommand{\qed}{$\square$}

\title{Addition law structure of elliptic curves}
\author{David Kohel\\
Institut de Math{\'e}matiques de Luminy\\
Universit{\'e} de la M{\'e}diterran{\'e}e\\
163, avenue de Luminy, Case 907\\
13288 Marseille Cedex 9\\ France}
\date{}

\begin{document}

\maketitle

\begin{abstract}
The study of alternative models for elliptic curves has found recent 
interest from cryptographic applications, after it was recognized 
that such models provide more efficiently computable algorithms for 
the group law than the standard Weierstrass model. 
Examples of such models arise via symmetries induced by a rational 
torsion structure.  We analyze the module structure of the space of 
sections of the addition morphisms, determine explicit dimension 
formulas for the spaces of sections and their eigenspaces under the 
action of torsion groups, and apply this to specific models of 
elliptic curves with parametrized torsion subgroups.
\end{abstract}


\section{Introduction}
\label{section:introduction}

Let $k$ be a field and $A$ an abelian variety over $k$ with a given projectively 
normal embedding $\iota: A \rightarrow \PP^r$, determined by an invertible sheaf 
$\cO_A(1) := \iota^*\cO_{\PP^r}(1)$ 
and denote the addition morphism on $A$ by
$
\mu: A \times A \rightarrow A.
$

An {\it addition law} is an $(r+1)$-tuple $\fs = (p_0,\dots,p_r)$ of bihomogeneous 
elements $p_j$ of 
$$
k[A] \otimes k[A] = k[X_0,\dots,X_r]/I_A \otimes_k k[X_0,\dots,X_r]/I_A,
$$
where $I_A$ is the defining ideal of $A$, such that the rational map
$$
(x,y) = ((x_0:\dots:x_r),(y_0:\dots:y_r)) \longmapsto (p_0(x,y):\dots:p_r(x,y))
$$
defines $\mu$ on the complement of $Z = V(p_0,\dots,p_r)$ in $A \times A$. 
The set $Z$ is called the {\it exceptional set} of $\fs$.  
Lange and Ruppert~\cite{LangeRuppert} give a characterization of addition laws, 
as sections of an invertible sheaf, from which it follows that the exceptional 
set of any nonzero addition law is the support of a divisor, which we refer to 
as the {\it exceptional divisor}. 
An addition law is said to have bidegree $(m,n)$ if $p_j(x,y)$ are homogeneous 
of degree $m$ and $n$ in $x_i$ and $y_j$, respectively.  
The 
addition laws of bidegree $(m,n)$, including the zero element, form  
a $k$-vector space. 

A set $S$ of addition laws is said to be {\it complete} or 
{\it geometrically complete} if the intersection of the exceptional sets of 
all $\fs$ in $S$ is empty, and {\it $k$-complete} or {\it arithmetically complete} 
if this intersection contains no $k$-rational point.  
We note that the term {\it complete}~\cite{BosmaLenstra,LangeRuppert,LangeRuppert-elliptic} 
has more recently been used to denote $k$-complete, in literature with a view 
to computational and cryptographic application.
The intersection of the exceptional sets for $\fs$ in $S$ clearly equals the 
intersection of the exceptional sets for all $\fs$ in its $k$-linear span.

The structure of addition laws depends intrinsically not just on $A$, 
but also on the embedding $\iota: A \rightarrow \PP^r$, determined by 
global sections $s_0, \dots, s_r$ in $\Gamma(A,\cL)$, for the sheaf 
$\cL = \cO_A(1)$. The hypothesis that $\iota$ is a projectively normal 
embedding may be defined to be the surjectivity of the homomorphism
$$
k[X_0,\dots,X_r] = 
\bigoplus_{n=0}^{\infty} \Gamma(\PP^r,\cO_{\PP^r}(n))
\longrightarrow 
\bigoplus_{n=0}^{\infty} \Gamma(A,\cL^n)
$$
(see Birkenhake-Lange~\cite[Chapter~7, Section~3]{BirkenhakeLange} or 
Hartshorne~\cite[Chapter~I, Exercise~3.18 \& Chapter~II, Exercise~5.14]{Hartshorne}). 
In particular, it implies that $\{s_0,\dots,s_r\}$ span $\Gamma(A,\cL)$. 
For an elliptic curve, the surjectivity of $\Gamma(\PP^r,\cO_{\PP^r}(1))$ 
on $\Gamma(E,\cL)$ is a necessary and sufficient condition for 
$\iota$ to be projectively normal.
We recall that an invertible sheaf is said to be {\it symmetric} if 
$\cL \isom [-1]^*\cL$. Lange and Ruppert~\cite{LangeRuppert} determine 
the structure of addition laws, and in particular prove the following 
main theorem. 

\begin{theorem}[Lange-Ruppert]
\label{theorem:Lange-Ruppert}
Let $\iota : A \rightarrow \PP^r$ be a projectively normal embedding of $A$,
and $\cL = \cO_A(1)$.  The sets of addition laws of bidegrees $(2,3)$ and 
$(3,2)$ on $A$ are complete.  
If $\cL$ is symmetric, then the set of addition laws of bidegree $(2,2)$ 
is complete, and otherwise empty.
\end{theorem}

\noindent{\bf Remark.}
Lange and Ruppert assume that $\iota$ is defined with respect to the 
complete linear system of an invertible sheaf $\cL \isom \cM^m$ where 
$\cM$ is ample and $m \ge 3$.  
Their hypothesis implies the projective normality of $\iota$ by a result 
of Sekiguchi~\cite{Sekiguchi} and the latter is sufficient for their proof.
Following Sekiguchi, Lange and Ruppert require that $k$ be algebraically 
closed, but the result relies only on the dimensions of sections of a 
certain line bungle and base-point freeness of its sections, which are 
independent of the base field.  We avoid this dependence by the direct 
assumption that $\iota$ is projectively normal.
\vspace{1mm}

Bosma and Lenstra~\cite{BosmaLenstra} give a precise description of the 
exceptional divisors of addition laws of 
bidegree $(2,2)$ when $A$ is an elliptic curve embedded as a 
Weierstrass model.  Using this analysis, they prove that two addition 
laws are sufficient for a complete system.  However, their description 
of the structure of addition laws applies more generally to other 
projective embeddings of an elliptic curve.  We carry out this analysis 
to determine the dimensions of spaces of addition laws in families with 
rational torsion subgroups and study the module decomposition of these 
spaces with respect to the action of torsion.

In view of Theorem~\ref{theorem:Lange-Ruppert}, the simplest possible structure 
of an addition law we might hope for is one for which the polynomials 
$p_j(x,y)$ are binomials of bidegree $(2,2)$.  
Such addition laws are known for 
Hessian models~\cite[Section~4]{ChudnovskyBrothers},~\cite{JoyeQuisquater-Hessian},~\cite{Smart-Hessian}, 
Jacobi quadric intersections~\cite[Section~4]{ChudnovskyBrothers}, and 
for Edwards models~\cite{BernsteinLange-Edwards},~\cite{Edwards} 
of elliptic curves.  
After recalling some background in Sections~\ref{section:sheaves} 
and~\ref{section:equivalence_and_dimensions}, and proving results 
about the exceptional divisors of addition laws, we introduce 
the concept of addition law projections in 
Section~\ref{section:addition_law_projections}. 
In Section~\ref{section:affine_models} we introduce the notion 
of a projective normal closure of an affine model of an elliptic 
curve in order to apply the preceding theory. 
Section~\ref{section:affine_addition_laws} gives a formal 
definition and interpretation of affine addition laws, 
expressed by rational functions, in terms of the addition law 
projections of Section~\ref{section:addition_law_projections}. 
In Section~\ref{section:torsion-module} we introduce a $G$-module 
structure of addition laws, with respect to a rational torsion 
subgroup on $E$.  In the final section we give examples of addition 
laws, observing that the simple laws coincide with the uniquely 
determined one-dimensional eigenspaces for the $G$-module structure.
In the final section we analyze the $G$-model structure of addition 
laws for standard families -- the degree 3 twisted Hessian models, 
the Jacobi quadric intersections and twisted Edwards models of 
degree 4 -- and construct an analogous degree 5 model for curves 
with a rational $5$-torsion structure. 

\section{Divisors and invertible sheaves on abelian varieties}
\label{section:sheaves}

Let $A/k$ be an abelian variety. We denote the addition morphism 
by $\mu$, the difference morphism by $\delta$, and let $\pi_i : 
A \times A \rightarrow A$ be the projection maps, for $i$ in $\{1,2\}$. 
We denote by $\mu^*$, $\delta^*$, and $\pi_i^*$ the respective 
pullback morphisms of divisors and sheaves from $E$ to $E \times E$.

We use the bijective correspondence between Weil divisors and Cartier 
divisors on abelian varieties, and to such a divisor $D$ we associate 
an invertible subsheaf $\cL(D)$ of the sheaf $\cK$ of total quotient 
rings such that for $D$ effective, $\cL(D)^{-1}$ is the ideal sheaf 
of $D$ (see Hartshorne~\cite[Chapter~II, Section~6]{Hartshorne}).   
\addition{
We call the invertible sheaf $\cL$ {\it effective} if it is isomorphic 
to $\cL(D)$ for some effective divisor $D$. 
}

For $\cL(D)$ so defined, its space of global sections is the Riemann-Roch 
space: 
$$
\Gamma(A,\cL(D)) = \{ f \in k(A) \;:\; \mathrm{div}(f) \ge D \},
$$
and an embedding $A \rightarrow \PP^r$ given by the complete linear 
system $|\cL(D)|$ is determined by 
$$
P \longmapsto (x_0(P):x_1(P):\dots:x_r(P)),
$$
for a choice of basis $\{x_0,x_1,\dots,x_r\}$ of $\Gamma(A,\cL(D))$.  
If $D$ is an effective Weil divisor we may take $x_0 = 1$, in which 
case we recover $D$ as the intersection with the hyperplane $X_0 = 0$ 
in $\PP^r$. 


\subsection{Sheaves associated to the addition morphism}

Lange and Ruppert~\cite{LangeRuppert} interpret an addition law of 
bidegree $(m,n)$ as a homomorphism of sheaves $\mu^*\cL \rightarrow 
\pi_1^*\cL^m \otimes \pi_2^*\cL^n$, then use the identification 
$$
\Hom(\mu^*\cL,\pi_1^*\cL^m \otimes \pi_2^*\cL^n) = 
\Gamma(A \times A, \mu^*\cL^{-1} \otimes \pi_1^*\cL^m \otimes \pi_2^*\cL^n),
$$
to determine their structure.  In view of Theorem~\ref{theorem:Lange-Ruppert}, we 
will be interested in symmetric invertible sheaves $\cL$, and the structure 
of sections of the sheaves
$$
\cM_{m,n} = \mu^*\cL^{-1} \otimes \pi_1^* \cL^m \otimes \pi_2^* \cL^n.
$$
and for the critical case of $\cM_{2,2}$ we write more concisely $\cM$.

\addition{
We return to the study of sheaves on $E \times E$ after characterizing 
certain properties of invertible sheaves and morphisms of elliptic curves.
}

\subsection{Invertible sheaves on elliptic curves}

A Weierstrass model of an elliptic curve $E$ with base point $\oO$ 
is determined with respect to $\cL(3(\oO))$ and any other cubic model 
in $\PP^2$ is obtained as a projective linear automorphism of the 
Weierstrass model.  As a prelude to the study of models determined 
by more general symmetric divisors, we recall the characterization 
of divisors on an elliptic curve.  For a divisor $D$ on an elliptic 
curve let $\eval(D)$ be its evaluation on the curve. 
With this notation, the following lemma is immediate.

\begin{lemma}
\label{lemma:elliptic-sheaf-canonical}
Let $\cL = \cL(D)$ be an invertible sheaf of degree $d$ on $E$.  
Then $\cL \isom \cL((d-1)(\oO)+(P))$ where $P = \eval(D)$. 
Moreover $\cL$ is symmetric if and only if $P$ is in $E[2]$.
\end{lemma}

\addition{
The classification of curves and their addition laws makes use 
of linear isomorphisms between spaces of global sections of an  
invertible sheaf.  In classifying curves and their addition laws,
it therefore makes sense to classify elliptic curves up to 
projective linear isomorphism.

\begin{lemma}
\label{lemma:linear_isomorphism}
Let $E_1$ and $E_2$ be projectively normal embeddings of an elliptic 
curve $E$ defined with respect to divisors $D_1$ and $D_2$. 
Then there exists a projective linear isomorphism $E_1 \rightarrow E_2$ 
if and only if $\deg(D_1) > \deg(D_2)$ or $D_1 \sim D_2$.
\end{lemma}

\begin{proof}
An equivalence of divisors $D_1 \sim D_2$ implies $\cL(D_1) \isom 
\cL(D_2)$, and the resulting linear isomorphism of global sections 
induces a linear isomorphism of the embeddings of the curve with 
respect to $D_1$ and $D_2$ (and whose inverse is also linear).  
If $\deg(D_1) > \deg(D_2)$, we may suppose -- up to equivalence 
-- that $D_1 > D_2 > 0$, and we have an inclusion of vector 
subspaces of $k(E)$:
$$
V_2 = \Gamma(E,\cL(D_2)) \subseteq V_1 = \Gamma(E,\cL(D_1))
$$
such that the restriction from $V_1$ to $V_2$ determines the 
morphism $E_1 \rightarrow E_2$ induced by a surjective linear 
map on coordinate functions.  Since $V_2$ defines the embedded 
image $E_2$, the restriction morphism is an isomorphism.
\qed
\end{proof}

A symmetric embedding gives rise to addition laws of minimal bidegree
in Theorem~\ref{theorem:Lange-Ruppert}. However, the structure of 
the negation map imposes additional motivation for requiring a 
symmetric line bundle.

\begin{lemma}
\label{lemma:symmetric-linear-inversion}
If $E \subset \PP^r$ is a projectively normal embedding with respect 
to $\cL$, then $[-1]$ is induced by a projective linear automorphism 
if and only if $\cL$ is symmetric.
\end{lemma}

\begin{proof}
If $\cL$ is symmetric, then $[-1]^*$ induces an automorphism of 
the space of global sections of $\cL$.  Conversely, since $E$ 
is projectively normal in $\PP^r$, a linear automorphism of the 
coordinate functions which determines $[-1]$ also induces an 
automorphism of global sections, hence of $\cL$ with $[-1]^*\cL$.
\qed
\end{proof}

In Section~\ref{section:torsion-module} we analyze the $G$-module 
structure of addition laws with respect to a finite subgroup 
$G = \{T_i\}$ of rational points on $E$.  For a rational point 
$T$ of $E$ we denote by $\tau_T$ the translation-by-$T$ map on $E$.
The following lemma characterizes when $\tau_T$ acts linearly.

\begin{lemma}
\label{lemma:linear_group_action}
Let $E \subset \PP^r$ be a projectively normal embedding with respect 
to $\cL$, and let $T$ be in $E(k)$. Then $\tau_T$ is induced by a 
projective linear automorphism if and only if $[\deg(\cL)]\,T = \oO$. 
\end{lemma}

\begin{proof}
It is necessary and sufficient to show that $\cL \isom \tau_T^*\cL$.
Let $\cL \isom \cL(D)$ and set $D' = \tau_T^*D$.  Since $\deg(D') = 
\deg(D)$ and $\eval(D') = \eval(D) - [\deg(D)]\,T$, by the canonical 
form of Lemma~\ref{lemma:elliptic-sheaf-canonical} the equivalence 
of the isomorphism $\cL \isom \tau_T^*\cL$ holds if and only if 
$[\deg(D)]\,T = \oO$.
\qed
\end{proof}

\noindent{\bf Remark.} We note that Lemma~\ref{lemma:linear_isomorphism} 
and Lemma~\ref{lemma:symmetric-linear-inversion} refer to isomorphisms 
in the category of elliptic curves (fixing a base point), while the 
isomorphism of Lemma~\ref{lemma:linear_group_action} is not an elliptic 
curve isomorphism. 
Lemma~\ref{lemma:linear_isomorphism} is false if an isomorphism in the 
category of curves is allowed.  Suppose that $E_1$ and $E_2$ are 
embedded with respect to divisors $D_1$ and $D_2$ and that 
$D_1 \sim \tau_T^* D_2$.  Then the morphism $\tau_T$ determines a 
linear isomorphism $E_1 \rightarrow E_2$ sending $\oO$ to $T$. 
}

\subsection{Invertible sheaves on $E \times E$}

Let $\mu$, $\delta$, $\pi_1$, and $\pi_2$ be the addition, difference, 
and projection morphisms, as above.  We define 
$$
V = \{\oO\} \times E \mbox{ and } H = E \times \{\oO\}
$$ 
as divisors on $E \times E$.  Similarly, let $\Delta$ and $\nabla$ be the 
diagonal and anti-diagonal images of $E$ in $E \times E$, respectively.  
\begin{lemma}
With the above notation we have
$$
\begin{array}{ll}
\pi_1^*\cL((\oO)) = \cL(V), \quad & \pi_2^*\cL((\oO)) = \cL(H), \\[2mm]
\mu^*\cL((\oO)) = \cL(\nabla), \quad & \delta^*\cL((\oO)) = \cL(\Delta).
\end{array}
$$
In particular if $\cL = \cL(d(\oO))$, then
$$
\mu^*\cL^{-1} \otimes \pi_1^* \cL^m \otimes \pi_2^* \cL^n = \cL(-d\nabla + dmV + dnH).
$$
\end{lemma}

\begin{proof}
This is immediate from 
$$
\mbox{\hspace{16mm}}
V = \pi_1^*(\oO),\ 
H = \pi_2^*(\oO),\  
\nabla = \mu^*(\oO) \mbox{ and }
\Delta = \delta^*(\oO). 
\mbox{\hspace{16mm} \qed}
$$
\end{proof}

We note that each of $V$, $H$, $\nabla$, and $\Delta$ is an elliptic 
curve isomorphic to $E$.  In the generalization of the divisor on $E$ 
from $3(\oO)$ to a more general Weil divisor, we obtain translates of 
these elementary divisors, which motivates the definitions 
$$
\begin{array}{ll}
\nabla_P := \mu^*(P) = \nabla + (P,\oO) = \nabla + (\oO,P),   & 
\,V_P := \pi_1^*(P) = V + (P,\oO), \\[1mm] 
\Delta_P := \delta^*(P) = \Delta + (P,\oO) = \Delta - (\oO,P), & 
H_P := \pi_2^*(P) = H + (\oO,P). 
\end{array}
$$
\addition{
For points $Q$ and $R$ in $E(\bar{k})$, let $\tau_{Q}$ and $\tau_{(Q,R)}$ 
be the translation morphisms on $E$ and $E \times E$. The following lemma 
is immediate from the definitions. 
\begin{lemma}
\label{lemma:divisor_translations}
The translation morphism $\tau_{(Q,R)}$ on $E \times E$ acts by pullback 
on divisors by:
$$
\begin{array}{lcl}
\displaystyle \tau_{(Q,R)}^*(\Delta_P) = \Delta_{P-Q+R}, & \quad &
\displaystyle \tau_{(Q,R)}^*(\,V_P\,) = V_{P-Q}, \\
\displaystyle \tau_{(Q,R)}^*(\nabla_P) = \nabla_{P-Q-R}, & \quad &
\displaystyle \tau_{(Q,R)}^*(H_P) = H_{P-R}. \\
\end{array}
$$
\end{lemma}
}

\subsection{Addition laws of bidegree $(2,2)$}

We now classify the sheaves of addition laws of bidegree $(2,2)$.  
We recall the definition of the invertible sheaf 
$$
\cM = \mu^*\cL^{-1} \otimes \pi_1^*\cL^2 \otimes \pi_2^*\cL^2.
$$

Following Bosma and Lenstra~\cite{BosmaLenstra}, we let $x$ be a 
degree 2 function on $E$ with poles only at $\oO$, and observe that for 
$x_1 = x \otimes 1$ and $x_2 = 1 \otimes x$ in $k(E) \otimes_k k(E) 
\subset k(E \times E)$, we have
$$
\mathrm{div}(x_1 - x_2) = \nabla + \Delta - 2V - 2H.
$$
\addition{
This relation gives rise to the following more general systems of relations.
\begin{lemma}
\label{lemma:divisor-equivalence}
For points $P$ and $P$ in $E(\bar{k})$ we have 
$$
\Delta_{P-Q} + \nabla_{P+Q} \sim V + V_{2P} + H + H_{2Q}.
$$
If $T_1$ and $T_2$ are in $E[2]$ and $T_3 = T_1 + T_2$, we have 
$$
\Delta_{T_1} + \nabla_{T_2} \sim V + V_{T_3} + H + H_{T_3}.
$$
\end{lemma}

\begin{proof}
The first relation is the homomorphic image of $\nabla + \Delta \sim 2V + 2H$ 
under $\tau_{(P,Q)}^*$, applying Lemma~\ref{lemma:divisor_translations}, 
then using the equivalences $2V_P \sim V + V_{2P}$ and $2H_P \sim H + H_{2P}$, 
which follow from the pullbacks of the sheaf isomorphisms of 
Lemma~\ref{lemma:elliptic-sheaf-canonical}.
The second relation follows by taking $S_1$ and $S_2$ such that $2S_i = T_i$,
and specializing to $(P,Q) = (-S_1 + S_2,S_1 + S_2)$.
\qed
\end{proof}
}

The above lemma yields the following isomorphisms in terms of symmetric 
invertible sheaves.
\begin{lemma}
\label{lemma:sheaf-isom}
Let $\cL$ be a symmetric invertible sheaf on $E$, let $T_1$ and $T_2$ be 
points in $E[2]$ and set $T_3 = T_1 + T_2$.  
The sheaves $\cL_i = \tau_{T_i}^*(\cL)$ satisfy 
$$
\mu^*\cL_{1} \otimes \delta^*\cL_{2} \isom 
\pi_1^*\cL \otimes \pi_1^*\cL_{3} \otimes 
\pi_2^*\cL \otimes \pi_2^*\cL_{3}, 
$$
and in particular
$$
\mu^*\cL \otimes \delta^*\cL \isom \pi_1^*\cL^2 \otimes \pi_2^*\cL^2, 
$$
from which $\cM \isom \delta^*\cL$.
\end{lemma}

\begin{proof}
By Lemma~\ref{lemma:elliptic-sheaf-canonical}, we have $\cL \isom \cL((d-1)(\oO)+(T))$ 
for some point $T$ in $E[2]$, and hence $\cL^2 \isom \cL(2d(\oO))$, and similarly for 
the translates $\cL_i$. The lemma then follows by the equivalences of 
Lemma~\ref{lemma:divisor-equivalence}, extended linearly to the pullbacks of divisors 
of the form $(d-1)(\oO)+(T)$.
\qed
\end{proof}

\noindent
The following theorem extends the analysis of Bosma and 
Lenstra~\cite[Section~4]{BosmaLenstra}, following the lines of proof of 
Lange and Ruppert~\cite[Section~2]{LangeRuppert} and \cite{LangeRuppert-elliptic}.

\begin{theorem}
\label{theorem:DeltaIsomorphism}
Let $\iota: E \rightarrow \PP^r$ be a projectively normal embedding of 
an elliptic curve, with respect to a symmetric sheaf $\cL \isom \cL(D)$. 
Then the space of global sections of $\cM$ is isomorphic to the space 
of global sections of $\cL$.  Moreover, the exceptional divisor of an 
addition law of bidegree $(2,2)$ associated to a section in 
$\Gamma(E \times E, \cM)$ is of the form $\sum_{i = 1}^d \Delta_{P_i}$ 
where $D \sim \sum_i (P_i)$.
\end{theorem}

\begin{proof}
In view of Lemma~\ref{lemma:sheaf-isom}, and since $\delta$ has integral 
fibers, we deduce that the difference morphism induces an isomorphism
$
\delta^*: \Gamma(E, \cL) \rightarrow \Gamma(E \times E, \delta^*\cL).
$
The structure of the exceptional divisor follows since for 
$D \sim \sum_i (P_i)$, we have $\delta^* D \sim \sum_i \Delta_{P_i}$.
\qed
\end{proof}

Since each $\Delta_{P_i}$ is isomorphic to $E$ over the algebraic closure 
of $k$, this theorem gives a simple characterization of the exceptional divisor, 
and of arithmetic completeness.

\begin{corollary}
\label{corollary:exceptional-intersection}
The exceptional divisor of an addition law of bidegree $(2,2)$ is 
of the form $C = \delta^*(D')$ where $C\,\cap\,H = D' \times \{\oO\}$.
\end{corollary}

\begin{proof}
Each component of $C$ is of the form $\Delta_P = \delta^*(P)$ for a 
uniquely determined $P$, and the identity $\Delta_P\,\cap\,H = (P,\oO)$ 
extends linearly to general sums of divisors of the form $\Delta_P$.
\qed
\end{proof}

\begin{corollary}
\label{corollary:exceptional-completeness}
An addition law of bidegree $(2,2)$ with exceptional divisor $C = \delta^*(D')$ 
is $k$-complete if and only if $D'$ has no $k$-rational point in its support. 
\end{corollary}

\begin{proof}
A component $\Delta_P$ of $C$ has a rational point (and is isomorphic to $E$) 
if and only if the point $P$ lies in $E(k)$.
\qed
\end{proof}

\addition{
\noindent{\bf Remark.} 
For addition laws of bidegree $(2,2)$, 
Corollary~\ref{corollary:exceptional-intersection} gives an elementary 
algorithm for characterizing the exceptional divisor and 
Corollary~\ref{corollary:exceptional-completeness} for characterizing 
arithmetic completeness. 
}

\section{Divisors and intersection theory}
\label{section:equivalence_and_dimensions}

For higher bidegrees, we do not expect to have an isomorphism 
between the space addition laws and the sections of an invertible 
sheaf on $E$. In order to determine the dimensions of these spaces, 
we require an explicit determination of the Euler-Poincar{\'e}
characteristic $\chi(E \times E,\cL)$ as a tool for determining 
the dimension of 
$
\Gamma(E \times E,\cL) = H^0(E \times E,\cL).
$

\subsection{Euler-Poincar{\'e} characteristic and divisor equivalence}

For a projective variety $X/k$ and a sheaf $\cF$, and let $\chi(X,\cF)$ 
be the Euler-Poincar{\'e} characteristic:
$$
\chi(X,\cF) = \sum_{i=0}^{\infty} (-1)^i \dim_k(H^i(X,\cF)). 
$$
For the classification of divisors or invertible sheaves of $X$, 
we have considered the {\it linear equivalence} classes in $\Pic(X)$.  
In order to determine the dimensions of spaces of addition laws, 
it suffices to consider the coarser {\it algebraic equivalence} class 
in the {\it N{\'e}ron-Severi group} of $X$, defined as
$$
\NS(X) = \Pic(X)/\Pic^0(X).
$$
For a surface $X$, a divisor $D$ is {\it numerically equivalent} 
to zero if the intersection product $C.D$ is zero for all curves 
$C$ on $X$.  This gives the coarsest equivalence relation on $X$ 
and we denote the group of divisors modulo numerical equivalence 
by $\Num(X)$.  We refer to Lang~\cite[ Chapter~IV]{Lang} for the 
general definition of $\Num(X)$, and the equality between $\Num(X)$ 
and $\NS(X)$ for abelian varieties:

\begin{lemma}
If $X$ is an abelian variety then $\NS(X) = \Num(X)$.
\end{lemma}

By the definition of numerical equivalence, the intersection product is 
nondegenerate on $\Num(X)$. In the application to $X = E \times E$,
we can determine the structure of $\NS(X)$.

\begin{lemma}
\label{lemma:NS-exact-sequence}
The following diagram is exact.
$$
\SelectTips{cm}{10}
\xymatrix@R=4mm{
         & 0 \ar[d] & 0 \ar[d] \\
0 \ar[r] & \Pic^0(E) \times \Pic^0(E) \ar[r] \ar[d] & \Pic^0(E \times E) \ar[r] \ar[d] & 0 \ar[d] \\
0 \ar[r] &  \Pic(E)  \times  \Pic(E)  \ar[r]^{\quad \pi_1^* \times \pi_2^*} \ar[d] &  \Pic(E \times E)  \ar[r] \ar[d] & \End(E) \ar[r] \ar[d] & 0\\
0 \ar[r] &  \NS(E)   \times  \NS(E)   \ar[r] \ar[d] &  \NS(E \times E)   \ar[r] \ar[d] & \End(E) \ar[r] \ar[d] & 0 \\
         & 0 & 0 & 0 \\
}
$$
\end{lemma}

\begin{proof}
Exactness of the middle horizontal sequence is Exercise IV 4.10 of 
Hartshorne~\cite{Hartshorne}, and the vertical sequences are exact
by the definition of the N{\'e}ron-Severi group.  Exactness of the upper 
and lower sequences follows by commutativity of the diagram.
\qed
\end{proof}

We note that since $\NS(E)$ and $\End(E)$ are free abelian groups, 
the lower sequence splits, with the splitting sending an endomorphism 
$\varphi$ to its graph $\Gamma_\varphi$, \addition{
where, in particular, $\Gamma_{[1]}= \Delta$ and $\Gamma_{[-1]} = \nabla$. 
Moreover, $E \times E$ is isomorphic to $\Pic^0(E\times E)$, with 
isohomomorphism $(P,Q) \mapsto V_P - V + H_Q - H$.}

Summarizing arguments from Lange and Ruppert~\cite{LangeRuppert-elliptic}, 
particularly the proof of Lemma~1.3, we now determine the intersection 
pairing on $\NS(E \times E)$. 
 
\begin{lemma}
\label{lemma:NS-intersection}
The N\'eron-Severi group $\NS(E \times E)$ is a finitely generated free 
abelian group, and if $\End(E) \isom \ZZ$, it is generated by $V$, $H$, 
$\Delta$, and $\nabla$, modulo the relation 
$
\Delta + \nabla \equiv 2V + 2H.  
$
The intersection product is nondegenerate on $\NS(E \times E)$ and given by
$$
\begin{array}{c|c|c|c|c}
       & V & H & \Delta & \nabla \\ \hline
     V & 0 & 1 &    1   &   1    \\ \hline
     H & 1 & 0 &    1   &   1    \\ \hline
\Delta & 1 & 1 &    0   &   4    \\ \hline
\nabla & 1 & 1 &    4   &   0    
\end{array}
$$
\end{lemma}

\begin{proof}
The divisors $V$ and $H$ are the generators of $\pi_1^*(\NS(E))$ and $\pi_2^*(\NS(E))$. 
Since $\Delta$ and $\nabla$ are the graphs of $[1]$ and $[-1]$, their sum induces 
the zero homomorphism, thus must lie in the image of $\pi_1^* \times \pi_2^*$. 
The expression for $\Delta + \nabla$ follows from the linear equivalence relation 
of Lemma~\ref{lemma:divisor-equivalence}.
Each of $V$, $H$, $\Delta$ and $\nabla$ has trivial self-intersection, since they have 
trivial intersections with their translates in $E \times E$.  The identities
$$
V.H = V.\Delta = V.\nabla = H.\Delta = H.\Delta = 1,
$$
hold since each pair has a unique intersection point $(\oO,\oO)$, and finally 
$\Delta.\nabla = 4$ follows from 
$
|\Delta \cap \nabla| = |\{ (T,T) \,:\, T \in E[2] \}| = 4.
$
\qed
\end{proof}

In the case of complex multiplication, the generator set can be extended by 
additional independent divisors $\Gamma_{\varphi_1}, \dots, \Gamma_{\varphi_{r-1}}$,
where $\{1,\varphi_1,\dots,\varphi_{r-1}\}$ is a basis for $\End(E)$, by the 
splitting of the lower sequence of Lemma~\ref{lemma:NS-exact-sequence}.
\begin{theorem}
\label{theorem:numerical-ampleness}
Let $E$ be an elliptic curve and $\cL$ be an invertible sheaf on $E \times E$.
The Euler-Poincar{\'e} characteristic $\chi(E \times E,\cL)$ depends only on the 
numerical equivalence class of $\cL$, and in particular
$$
\chi(E \times E,\cL(D)) = \frac{1}{2} D.D.
$$
If $\cL$ is ample, then $\chi(E \times E,\cL) = \dim_k(\Gamma(E \times E,\cL))$. 
Conversely, if $\cL$ is effective and $\chi(E \times E,\cL)$ is positive, 
then $\cL$ is ample.
\end{theorem}

\addition{
\begin{proof}
The first statement is the Riemann-Roch theorem for abelian surfaces 
(see Mumford~\cite[p.~150]{Mumford} or 
Hartshorne~\cite[Chapter~V, Theorem~1.6]{Hartshorne}).
For the latter statements, Mumford's Vanishing Theorem~\cite[p.~150]{Mumford} 
states that when $\chi(E \times E,\cL)$ is nonzero, $H^i(E \times E,\cL) \ne 0$
for exactly one $i = i(\cL)$ and that $0 \le i \le 2$. 
In addition, $i(\cL) = i(\cL^n)$ for all $n > 0$~\cite[Corollary, p.~159]{Mumford}. 
If $\cL$ is ample it follows that $i = 0$, and since $H^0(E \times E,\cL) = 
\Gamma(E \times E,\cL)$ gives the only contribution to the Euler-Poincar{\'e} 
characteristic, the result follows.  
In the other direction, for positive Euler-Poincar{\'e} characteristic, clearly 
$i \ne 1$, and by Serre duality~\cite[Chapter~III, Corollary~7.7]{Hartshorne} we have 
$i(\cL^{-1}) = 2 - i(\cL)$. 
For $\cL$ effective, $H^0(E \times E,\cL^{-1}) = 0$, hence $i = 0$.
Ampleness of $\cL$ follows by Application~1 of Mumford~\cite[p.~60]{Mumford}.
\qed
\end{proof}
}

The following corollary of Theorem~\ref{theorem:numerical-ampleness}
and Lemma~\ref{lemma:NS-intersection}, which is synthesis of results 
of Lange and Ruppert~\cite{LangeRuppert,LangeRuppert-elliptic}, 
allows the effective determination of the Euler-Poincar{\'e} 
characteristic.

\begin{corollary}[Lange-Ruppert]
\label{corollary:EulerCharacteristic}
Let $E$ be an elliptic curve, then 
$$
\chi(E \times E,\cL(x_0\nabla + x_1V + x_2H)) = x_0 x_1 + x_0 x_2 + x_1 x_2.
$$
In particular, if $\cL$ is an invertible sheaf of degree $d > 0$ on $E$, then
$$
\chi(E \times E, \cM_{m,n}) = d^2(mn - m - n).
$$
\end{corollary}

\addition{
As an application, we have a clear criterion for the sheaves $\cM_{m,n}$ 
to be effective and ample.  We define the product order in bidegrees by 
$(k,l) < (m,n)$ if and  only if $k < m$ and $l < n$.
}

\begin{corollary}
The sheaf $\cM_{m,n}$ is ample if and only if $(2, 2) < (m, n)$.
\end{corollary}

\addition{
\begin{proof}
The Euler-Poincar{\'e} characteristic, 
$\chi(E \times E, \cM_{m,n}) = d^2(mn-m-n)$, is positive if and only if 
$(2, 2) < (m, n)$ by Corollary~\ref{corollary:EulerCharacteristic}, and 
this is a necessary condition for ampleness, e.g.~by the Nakai-Moishezon 
Criterion (see Hartshorne~\cite[Chapter~V, Theorem 1.10]{Hartshorne}).
On the other hand, $\cM_{m,n}$ is isomorphic to 
$$
\cL(d\Delta+d(m-2)V+d(n-2)H),
$$
hence is effective when $(2, 2) < (m, n)$, so $\cM_{m,n}$ is ample by 
Theorem~\ref{theorem:numerical-ampleness}. 
\qed
\end{proof}

Next we obtain a characterization of the critical case 
$\chi(E \times E,\cL(D)) = 0$.  In view of the roles of $\nabla$, $\Delta$, 
$V$, and $H$ in the divisor theory, we define $\Gamma_{(a,b)}$ to be 
the image of $E$ in $E \times E$ given by $P \mapsto (aP, bP)$, for $a$ 
and $b$ coprime, and for $(na,nb)$ set $\Gamma_{(na,nb)} = n^2 \Gamma_{(a,b)}$. 
We then have equivalent expressions 
$$
\Delta = \Gamma_{(1,1)},\ \nabla = \Gamma_{(1,-1)}, \ 
V = \Gamma_{(0,1)},\ H = \Gamma_{(1,0)}.
$$

\begin{lemma}
\label{lemma:gamma_ab_numerical_class}
The divisor $\Gamma_{(a,b)}$ is numerically equivalent to 
$$
-ab\nabla + (a^2+ab)V + (ab+b^2)H.
$$
\end{lemma}

\begin{proof}
The numerical equivalence class is determined by the intersection 
products 
$$
(\nabla.\Gamma_{(a,b)}, V.\Gamma_{(a,b)}, H.\Gamma_{(a,b)}) = ((a+b)^2, b^2, a^2),
$$
which agrees with that of the divisor $-ab\nabla + (a^2+ab)V + (ab+b^2)H$.
\qed
\end{proof}

\begin{theorem}
\label{theorem:numerically-critical}
A divisor $D$ on $E \times E$ satisfies $\chi(E \times E,\cL(D)) = 0$ if and 
only if $D$ is numerically equivalent to $n\, \Gamma_{(a,b)}$ for integers 
$n$, $a$ and $b$.
\end{theorem}

\begin{proof}
By Lemma~\ref{lemma:NS-intersection}, every divisor is numerically equivalent 
to one of the form $D = x_0 \nabla + x_1 V + x_2 H$.  
By Corollary~\ref{corollary:EulerCharacteristic}, the identity 
$\chi(E \times E, \cL(D)) = 0$ defines a conic 
$$
C: x_0 x_1 + x_0 x_2 + x_1 x_2 = 0
$$
in $\PP^2$, which has a parametrization $\PP^1 \rightarrow C$ given by 
$$
(a : b) \longrightarrow (-ab : a^2 + ab : ab + b^2),
$$
hence every triple $(x_0,x_1,x_2)$ satisfying $x_0 x_1 + x_0 x_2 + x_1 x_2 = 0$
is of the form $n(-ab, a^2+ab, ab+b^2)$ for integers $n$, $a$ and $b$. 
By Lemma~\ref{lemma:gamma_ab_numerical_class}, the divisor $D$ is numerically 
equivalent to $n\,\Gamma_{(a,b)}$. 
\qed
\end{proof}

}

\subsection{Dimensions of spaces of addition laws}

We are now in a position to relate the dimension of $\Gamma(E \times E,\cM_{m,n})$
to $\chi(E \times E, \cM_{m,n})$.  As a first step, we recall the statement of
the Riemann-Roch theorem for elliptic curves.

\begin{theorem}
\label{theorem:RiemannRoch}
If $\cL$ is an invertible sheaf of degree $d > 0$ on an elliptic curve $E$, 
then $\cL$ is ample and $\dim_k(\Gamma(E,\cL)) = d$.
\end{theorem}

\begin{corollary}
\label{corollary:Dimensions}
Let $\cL$ be a symmetric ample invertible sheaf of degree $d$ on an 
elliptic curve $E$ and 
$$
\cM_{m,n} = \mu^*\cL^{-1} \otimes \pi_1^*\cL^m \otimes \pi_2^*\cL^n.
$$
Then for $(m,n) = (2,2)$, 
$$
\dim_k(H^0(E \times E,\cM)) = \dim_k(H^1(E \times E,\cM)) = d,
$$
and for all other $m, n \ge 2$, 
$$
\dim_k(H^0(E \times E,\cM_{m,n})) = d^2(mn - m - n).
$$
\end{corollary}

\begin{proof}
Since $\cM_{m,n}$ is isomorphic to $\cL(C)$ for an effective divisor 
$C$, we have that 
$$
H^2(E \times E,\cM_{m,n}) \isom H^0(E \times E, \cM_{m,n}^{-1}) = 0
$$ 
by Serre duality~\cite[Chapter~III, Corollary~7.7]{Hartshorne}, since 
$\omega_A \isom \cO_A$ for any abelian variety 
$A$~\cite[Chapter~1, Lemma~(4.2)]{BirkenhakeLange}.  
The dimension of the first cohomology group of $\cM_{m,n}$ is then 
determined by the dimension of $H^0(E \times E,\cM_{m,n})$ and 
the Euler characteristic of Corollary~\ref{corollary:EulerCharacteristic}. 

For $(m,n) = (2,2)$, the dimension of $H^0(E \times E, \cM)$ is determined 
by Theorem~\ref{theorem:DeltaIsomorphism} and Theorem~\ref{theorem:RiemannRoch}, 
and for all higher bidegrees the sheaf $\cM_{m,n}$ is ample and 
$\chi(E \times E,\cM_{m,n})$ equals 
$\dim_k(H^0(E \times E,\cM_{m,n}))$ by 
Theorem~\ref{theorem:numerical-ampleness}.
\qed
\end{proof}

These dimension formulas will be generalized in 
Section~\ref{section:addition_law_projections}, after the introduction   
of the concept of an addition law projection. 

\subsection{Dimensions of sections of the ideal sheaf}

When $E$ is embedded as a cubic curve in $\PP^2$, the defining ideal 
sheaf $\cI_E$ of $E$ has no sections of degree $2$, which is to say 
that $\dim_k(\Gamma(\PP^2,\cI_E(2))) = 0$.  
However, a degree $4$ or higher divisor always includes quadratic 
defining relations.  This introduces an ambiguity in the representation 
of an addition law by polynomials. In what follows,
when $E$ is not contained in a hyperplane of $\PP^r$, the ideal sheaf 
contains no linear relations, and the degree $d$ equals $r + 1$, 
since a projective normal embedding is given by a complete linear 
system.

\begin{lemma}
\label{lemma:ideal_sheaf_dimension}
Let $E$ be an elliptic curve and $\iota : E \rightarrow \PP^r$ 
be a projectively normal embedding of degree $d$. Then for the 
ideal sheaf $\cI_E$, we have 
$$
\dim_k(\Gamma(\PP^r,\cI_E(n))) = \binom{n+r}{r} - nd.
$$
\end{lemma}

\begin{proof}
Let $\cL = \cO_E(1)$ and note that $\Gamma(\PP^r,\cO(n)) \rightarrow 
\Gamma(E,\cL^n)$ is surjective by hypothesis. Thus the dimension is 
determined by the number of monomials of degree $n$ in $r+1$ variables  
minus the dimension of the space $\Gamma(E,\cL^n)$.  This latter space 
has dimension $nd$ by Riemann-Roch, from which the result follows.
\qed
\end{proof}

The polynomial representatives for the coordinates of an addition law 
of bidegree $(m,n)$ are well defined only up to elements of 
$$
I_{m,n} = 
\Gamma(\PP^r,\cI_E(m)) \otimes \Gamma(\PP^r,\cO_{\PP^r}(n)) + 
\Gamma(\PP^r,\cO_{\PP^r}(m)) \otimes \Gamma(\PP^r,\cI_E(n)).  
$$
Since, for $d \ge 4$, the dimension of $\Gamma(\PP^r,\cI_E(2))$ is nonzero, 
the addition laws for any nonplanar model have nonunique representation by 
polynomials.  We make this more precise in the following corollary.

\begin{corollary}
An addition law of bidegree $(m,n)$ is represented by a coset of a 
vector space of polynomials whose dimension is
$$
(r+1)\left(\binom{m+r}{r} \binom{n+r}{r} - d^2 m n\right).
$$
\end{corollary}

\begin{proof}
The dimension of the vector space $I_{m,n}$ equals
$$
\binom{m+r}{r} \binom{n+r}{r} - d^2 m n,
$$
determined by Lemma~\ref{lemma:ideal_sheaf_dimension} and M\"obius 
inversion with respect to the common vector subspace
$
\Gamma(\PP^r,\cI_E(m)) \otimes \Gamma(\PP^r,\cI_E(n)).
$
Since each of the $r+1$ polynomials representing the addition law coordinates 
is a coset of the vector space $I_{m,n}$ we obtain the cofactor $r+1$.
\qed
\end{proof}

\section{Addition law projections}
\label{section:addition_law_projections}

We introduce the notion of an addition law projection first in order 
to define the concept of an affine addition law given by rational maps, 
expressed in terms of morphisms $E \times E \mapsto \PP^1$ which 
factor through $\mu$.  In addition we are able to consider generalizations 
of addition laws which take the form $E_1 \times E_1 \rightarrow E_2$, 
where $E_1$ and $E_2$ are different embeddings, defined by divisors 
$D_1$ and $D_2$.  


\subsection{Definition of an addition law projection}

Let $E$ be projectively normal in $\PP^r$ with $\cL = \cO_E(1)$, 
let $\varphi: E \rightarrow C \subset \PP^s$ be a morphism, and 
set $\cL_{\varphi} = \varphi^*\cO_{C}(1)$. We assume that 
$\cL \isom \cL(D)$ and 
$\cL_{\varphi} \isom \cL(D_\varphi)$.
We now define the space of {\it addition law projections} of bidegree 
$(m,n)$ with respect to the composition $\varphi \circ \mu$ to be the 
set of $(s+1)$-tuples $\fs = (p_0,\dots,p_s)$ with 
$$
p_j \in \Gamma(E \times E, \pi_1^*\cL^m \otimes \pi_2^*\cL^n),
$$
determining $\varphi \circ \mu$ on an open subvariety of $E \times E$.
As above, we interpret an addition law projection $\fs$ as an element of 
$\Hom(\mu^*\cL_\varphi, \pi_1^*\cL^m \otimes \pi_2^*\cL^n)$, isomorphic
to 
$$
\Gamma(E \times E, \mu^* \cL_\varphi^{-1} \otimes \pi_1^*\cL^m \otimes \pi_2^*\cL^n). 
$$
The principal interest is when, up to isomorphism, $D > D_\varphi > 0$, 
and $\varphi$ is either an isomorphism or a projection to $\PP^1$.  
In such a case, the morphism $\varphi$ has a linear representation and 
an addition law for $\mu$ restricts to an addition law projection for 
$\varphi \circ \mu$.  On the other hand, the space of addition laws 
projections is in general larger and may be nonzero for bidegrees less 
than $(2,2)$.

\subsection{Dimensions of spaces of addition law projections}

We are now in a position to determine the dimensions of the spaces of 
addition law projections.  Let $E$ be a projectively normal curve in 
$\PP^r$ with $\cL = \cO_E(1) \isom \cL(D)$ and $\varphi$ a nonconstant
morphism to a curve $C$ in $\PP^s$ such that 
$$
\cL_\varphi := \varphi^*\cO_C(1) \isom \cL(D_\varphi)
$$ 
and define 
$$
\cM_{\varphi,m,n} = \mu^*\cL_\varphi^{-1} \otimes \pi_1^*\cL^m \otimes \pi_2^*\cL^n.
$$
We assume $D$ and $D_\varphi$ are effective and write $d = \deg(D)$ 
and $d_{\varphi} = \deg(D_\varphi)$. With this notation, we obtain the 
following refinement of Corollary~\ref{corollary:EulerCharacteristic}, 
as a consequence of Lemma~\ref{lemma:NS-intersection} and 
Theorem~\ref{theorem:numerical-ampleness}.

\begin{corollary}
\label{corollary:phi-EulerCharacteristic}
$\chi(\cM_{\varphi,m,n}) = d (dmn - d_\varphi(m + n))$.
\end{corollary}

\addition{
When $d = 2d_\varphi$ the critical bidegree is $(1,1)$, for which the 
Euler-Poincar{\'e} characteristic is zero, and when $d \ge 2d_\varphi$, 
the minimal bidegree of any addition law projection is $(1,1)$, so we 
write simply $\cM_\varphi$ for the sheaf $\cM_{\varphi,1,1}$.  We can 
now state a generalization of Theorem~\ref{theorem:DeltaIsomorphism}. 

\begin{theorem}
\label{theorem:phi-DeltaIsomorphism}
Let $\iota: E \rightarrow \PP^r$ be a projectively normal embedding of 
an elliptic curve, with respect to a symmetric sheaf $\cL \isom \cL(D)$, 
and let $\varphi: E \mapsto \PP^s$ be a nonconstant map, with respect 
to a symmetric sheaf $\cL_\varphi\isom \cL(D_\varphi)$.  
If both $\cL$ and $\cL_\varphi$ are symmetric and $d = 2d_\varphi$, 
then the space of global sections of $\cM_{\varphi}$ is isomorphic 
to that of $\cL \otimes \cL_\varphi^{-1}$.  Moreover, the exceptional 
divisor of an addition law projection of bidegree $(1,1)$ associated 
to a section in $\Gamma(E \times E, \cM_{\varphi})$ is of the form 
$\sum_{i = 1}^{d_\varphi} \Delta_{P_i}$ where $D-D_\varphi \sim 
\sum_{i=1}^{d_\varphi} (P_i)$.
\end{theorem}

\begin{proof}
Up to equivalence of $D$ and $D_\varphi$, we may assume that $D$, $D_\varphi$,
and $D_\psi$ are symmetric effective divisors, hence with support in $E[2]$, 
such that $D = D_\varphi + D_\psi$, and denote $\cL(D_\psi)$ by $\cL_\psi$.
By hypothesis $d = 2\deg(D_\psi) = 2d_\varphi$. 
Linear extension of Lemma~\ref{lemma:divisor-equivalence} then gives
$$
\delta^*D_\psi + \mu^*D_\varphi \sim \pi_1^*D + \pi_2^*D,
$$
and hence
$$
\delta^*\cL_\psi \otimes \mu^*\cL_\varphi \isom \pi_1^*\cL \otimes \pi_2^*\cL,
$$
from which $\delta^*\cL_\psi \isom \cM$.  
The isomorphism of global sections and structure of the exceptional 
divisors follows as in Theorem~\ref{theorem:DeltaIsomorphism}. 
\qed
\end{proof}
}

\begin{corollary}
\label{corollary:phi-Dimensions}
Let $\cL$, $\cL_\varphi$, and $\cM_{\varphi,m,n}$ be as above, with 
$d = 2d_\varphi$.  Then for $(m,n) = (1,1)$, 
$$
\dim_k(H^0(E \times E,\cM_\varphi)) = \dim_k(H^1(E \times E,\cM_\varphi)) = d_\varphi,
$$
and for $(m,n) > (1,1)$, 
$$
\dim_k(H^0(E \times E,\cM_{\varphi,m,n})) = d_\varphi^2((2m-1)(2n-1)-1).
$$
\end{corollary}

\addition{
\begin{proof}
For $(m,n) = (1,1)$ the dimension follows from the isomorphism of 
Theorem~\ref{theorem:phi-DeltaIsomorphism}. When $(m,n) > (1,1)$, the 
sheaf $\cM$ is effective and $\chi(E\times E,\cM)$ positive, 
so the equality follows from Theorem~\ref{theorem:numerical-ampleness}
and Corollary~\ref{corollary:phi-EulerCharacteristic}.
\qed
\end{proof}

\noindent{\bf Remark.}
As a consequence of Corollary~\ref{corollary:phi-EulerCharacteristic}, 
the only possible critical cases, 
for which $\chi(E \times E,\cM_{\phi,m,n}) = 0$, are those 
in the table below, given with the value of 
$h^0 = \dim_k(\Gamma(E\times E,\cM_{\varphi,m,n}))$, if nonzero.
$$
\begin{array}{c|c|c|c|c}
d & d_\varphi & (m,n) & h^0 \\ \hline
s(t+1) & st & (1,t), (t,1) & s \\
2d_\varphi & d_\varphi & (1,1) & d_\varphi \\
 d_\varphi & d_\varphi & (2,2) & d_\varphi 
\end{array}
$$
The latter two cases are explained by 
Theorem~\ref{theorem:DeltaIsomorphism} and 
Theorem~\ref{theorem:phi-DeltaIsomorphism}.
By Theorem~\ref{theorem:numerically-critical}, the exceptional divisor 
is numerically equivalent to a divisor of the form $n \Gamma_{(a,b)}$. 
For $d = 2d_\varphi$ and $d = d_\varphi$, this divisor is $d_\varphi 
\Delta$, but for $(d,d_\varphi) = (s(t+1),st)$, the exceptional divisor 
is numerically equivalent to $s\Gamma_{(1,t)}$ or $s\Gamma_{(t,1)}$. 
Theorem~\ref{theorem:exotic_addition_laws} of 
Section~\ref{section:constructions} gives an example of an elliptic 
curve with one-dimensional spaces of addition laws projections of 
bidegrees $(1,2)$ and $(2,1)$ for $(d,d_\varphi) = (3,2)$. 
}

\section{Affine models and projective normal closure}
\label{section:affine_models}

A nonsingular projective curve is uniquely determined, up to unique 
isomorphism, by an affine model $C$~\cite[Chapter~I, Corollary~6.12]{Hartshorne}.  
As a consequence, it is standard to specify a curve by an affine model
which determines it.  On the other hand, the definition of addition 
laws in terms of a given affine model depends on the projections to 
$\PP^1$ given by the coordinate functions.  In this section we introduce 
the notion of a projective normal closure of a nonsingular affine model $C$. 
This provides a canonical nonsingular projective model in which $C$
embeds, in terms of which we define affine addition laws. 
In Section~\ref{section:affine_addition_laws} we apply this definition 
in order to determine dimension formulas for affine addition laws.

\subsection{Projective normal closure}

Let $C/k$ be a nonsingular affine curve in $\AA^s$, with coordinate 
functions $x_1,\dots, x_s$ and $X$ its associated nonsingular 
projective curve.  We define the {\it divisor at infinity} of $C$ 
to be the effective divisor
$
D = \sup(\{ \mathrm{div}_\infty(x_i) \}),
$
on $X$, where $\mathrm{div}_\infty(x)$ is the polar divisor of $x$.  

Let $\{ x_0, x_1, \dots, x_r \}$ be a generator set for $\Gamma(X,\cL(D))$, 
where we assume $x_0 = 1$, and $x_1, \dots, x_s$ are the coordinate 
functions on $C$.  Since $C$ is nonsingular, its coordinate ring is 
integrally closed, and by the definition of $D$, we have 
$$
k[x_1,\dots,x_s] = k[x_1,\dots,x_r].
$$
A {\it projectively normal closure} of $C$ is a model for $X$ in $\PP^r$,
determined by the morphism
$$
P \longmapsto (x_0(P):x_1(P):\dots:x_r(P)),
$$
which identifies $C$ as the open affine of $X$ given by $X_0 = 1$.
Any two projectively normal closures are isomorphic via a linear 
isomorphism determined by the choice of generator set extending 
$x_0,\dots,x_s$.
\vspace{2mm}

\addition{
\noindent{\bf Jacobi model.}
The Jacobi quartic refers to the nonsingular affine curve 
$$
y^2 = x^4 + 2ax^2 + 1,
$$
with base point $\oO = (0,1)$, and whose standard projective 
closure in $\PP^2$ is singular.  The divisor at infinity is 
$D = 2(\infty_1) + 2(\infty_2)$, and the Riemann-Roch space 
$\Gamma(E,\cL(D))$ is spanned by $\{1,x,y,x^2\}$.  Thus the 
projective normal closure is the curve $C$ in $\PP^3$ given 
by the embedding $(x,y) \mapsto (1:x:y:x^2)$, with defining 
equations 
$$
X_2^2 = X_0^2 + 2a X_0 X_3 + X_3^2, \quad 
    X_0 X_3 = X_1^2, 
$$
and identity $(1:0:1:0)$. 

The Jacobi quartic has full rational $2$-torsion, which 
accounts for the symmetries. 
In Section~\ref{section:constructions} we describe a 
canonical Jacobi model, diagonalized with respect to 
the $2$-torsion subgroup, and which contains this 
family as a subfamily up to linearly isomorphism.
\vspace{2mm}

\noindent{\bf Edwards model.}
In 2007, Edwards~\cite{Edwards} introduced a remarkable new affine 
model for elliptic curves 
$$
x^2 + y^2 = c^2(1 + dx^2y^2).
$$ 
The parameter $d$, equal to $1$ in Edwards' model, was introduced 
by Bernstein and Lange~\cite{BernsteinLange-Edwards}, to obtain 
an $k$-complete addition law for nonsquare values of $d$ (and moreover 
the parameter $c$ may be subsumed into $d$ as a square factor). 
Subsequently, Bernstein et al.~\cite{BernsteinEtAl-TwistedEdwards} 
introduced twisted Edwards curves 
$$
a x^2 + y^2 = 1 + dx^2y^2.
$$
\ignore{
On this model, there exists a $k$-complete affine addition law: 
$$
(x_1,y_1) + (x_2,y_2) \longmapsto \left(
\frac{x_1 y_2 + y_1 x_2}{1 + d x_1 y_1 x_2 y_2}, 
\frac{y_1 y_2 - a x_1 x_2}{1 - d x_1 y_1 x_2 y_2} \right)\cdot
$$
}

The divisor at infinity is $D = \mathrm{div}(x)_\infty + \mathrm{div}(y)_\infty$,
since the poles of $x$ and $y$ are disjoint. 
A basis for the Riemann-Roch space of $D$ is then $\{1,x,y,xy\}$, 
and the projective normal closure in $\PP^3$ is
$$
X_0^2 + dX_3^2 = aX_1^2 + X_2^2, \quad X_0 X_3 = X_1 X_2, 
$$
with embedding $(x,y) \mapsto (1:x:y:xy)$.  This embedding in 
$\PP^3$ appears in Hisil et al.~\cite{Hisil-EdwardsRevisited},
under the name extended Edwards coordinates.
\ignore{
The affine addition law now takes the bilinear form 
$$
(x_1,y_1,z_1) + (x_2,y_2,z_2) \longmapsto (x_3,y_3) = 
\left(
  \frac{x_1 y_2 + y_1 x_2}{1 + d z_1 z_2}, 
  \frac{y_1 y_2 - a x_1 x_2}{1 - d z_1 z_2} \right)\cdot
$$
In Section~\ref{section:affine_addition_laws} we provide the 
theoretical framework for studying such affine addition laws.
}
}

\subsection{Arithmetically complete affine models}

The notion of completeness of addition laws is sometimes coupled 
with an independent condition on a particular affine model.  
By definition an abelian variety is a complete group variety -- 
completeness is a geometric notion which is stable under base 
extension.  We define an affine curve $C$ to be {\it $k$-complete} 
or {\it arithmetically complete} if $C(k) = X(k)$ for any projective 
nonsingular $X$ containing $C$.  
For an elliptic curve, this ensures that the rational points of 
the affine model form a group.  Over a sufficiently large base 
field, one can find a suitable line which misses all rational 
points and pass to a $k$-complete affine model by a projective 
change of variables.  

\addition{
In the above example of a projective normal closure for 
the Jacobi quartic, the affine patch $X_2 = 1$: 
$$
1 = u^2 + 2 a u w + w^2, \quad  u w = v^2, 
$$
is $k$-complete if $x^2 + 2ax + 1$ is irreducible.  
The affine patch $X_3 = 1$ of the projective normal closure 
of the twisted Edwards model recovers the standard 
affine representation, which is $k$-complete when $d$ is 
a nonsquare. An additional feature of the $k$-complete models 
for twisted Edwards curves~\cite{BernsteinEtAl-TwistedEdwards} 
or twisted Hessian curves~\cite{BernsteinKohelLange} is that 
the line at infinity is an eigenvector for a torsion subgroup, 
which acts linearly on the affine curve.  
}

\section{Affine addition laws}
\label{section:affine_addition_laws}

Suppose that $C$ is a nonsingular affine curve in $\AA^s$ and 
let $E$ be a projective normal closure of $C$ in $\PP^r$.  
If $x_1, \dots, x_s$ are the coordinate functions on $C$, then 
we denote by $x_i$ also the projections 
$
E \rightarrow \PP^1
$ 
extending 
$
x_i : C \rightarrow \AA^1.
$
Let $k[C] = k[x_1,\dots,x_s]$ be the coordinate ring of $C$, recalling 
that since $C$ is nonsingular, 
$
k[x_1,\dots,x_s] = k[x_1,\dots,x_r] = \Gamma(C,\cO_E),
$
where $x_i = X_i/X_0$. We write 
$$
k[C] \otimes_k k[C] = k[x_1,\dots,x_r,y_1,\dots,y_r]
$$
where we identify $x_i$ with $x_i \otimes 1$ and write $y_i$ for $1 \otimes x_i$, 
and similarly identify $X_i$ with
$$
X_i \otimes 1 \in \Gamma(E \times E, \pi_1^*\cO_E(1) \otimes \pi_2^*\cO_E(0)), 
$$
and write $Y_i$ for 
$$
1 \otimes X_i \in \Gamma(E \times E, \pi_1^*\cO_E(0) \otimes \pi_2^*\cO_E(1)). 
$$
An affine addition law for $C$ is an $s$-tuple of pairs $(f_i,g_i)$ in 
$(k[C] \otimes_k k[C])^2$ such that 
$$
\mu^*(x_i) = \frac{f_i}{g_i} \in k(E \times E).
$$
We refer to $(f_i,g_i)$ as an {\it affine addition law projection} for $x_i$. 
We define the {\it bidegree} of an addition law $\fs_i = (f_i,g_i)$ to 
be the smallest $m_i$ and $n_i$ such that $\fs_i$ is the restriction of 
an addition law projection of bidegree $(m_i,n_i)$, and the bidegree of 
$\fs = (\fs_1,\dots,\fs_s)$ to be $(m,n) = ( \max_i(\{m_i\}), \max_i(\{n_i\}) )$.
We note that the bidegree of an addition law is determined by the minimal 
degree polynomial expression in $\{x_1,\dots,x_r,y_1,\dots,y_r\}$ for 
$f_i$ and $g_i$, rather than as a polynomial in the coordinate functions 
on $\{x_1,\dots,x_s,y_1,\dots,y_s\}$.

\addition{ 
Recall that the product partial order is defined by $(k,l) \le (m,n)$ 
if and only if $k \le m$ and $l \le n$.  
Clearly an addition law projection of bidegree $(k,l)$ is also the 
restriction of an addition law projection of bidegree $(m,n)$ when 
$(k,l) \le (m,n)$, since the restriction map associated to 
$C \rightarrow E$ is the homomorphism which forgets the grading:
$$
k[E] \isom \bigoplus_{n=0}^\infty \Gamma(E,nD) 
\longrightarrow 
k[C] = \Gamma(C,\cO_E) = \bigcup_{n=0}^\infty \Gamma(E,nD).
$$
For convenience, we say the space of addition 
laws (or addition law projections) of bidegree $(m,n)$, to refer to 
the vector space of all addition laws (or addition law projections)   
of any bidegree $(k,l) \le (m,n)$.}

Hereafter we express an affine addition law projection $(f_i,g_i)$ as a 
fraction $f_i/g_i$ and similarly write
$$
\fs = \left( \frac{f_1}{g_1}\ccomma\, 
             \frac{f_2}{g_2}\ccomma\, \cdots\ccomma\, 
             \frac{f_s}{g_s} \right)\!\ccomma
$$
for an affine addition law.  We note that in this context $f_i/g_i$ should 
not be confused with the equivalence class $z_i = \mu^*(x_i)$ in $k(E \times E)$,
and that in this notation the vector space structure is written:
$$
a \frac{f_i}{g_i} + b \frac{f_i'}{g_i'} = \frac{a f_i + b f_i'}{a g_i + b g_i'}\cdot 
$$
Since $f_i = g_i z_i$ and $f_i' = g_i' z_i$, the equivalence class in 
$k(E \times E)$ remains the same:
$$
a \frac{f_i}{g_i} + b \frac{f_i'}{g_i'} = 
a \frac{g_i z_i}{g_i} + b \frac{g_i' z_i}{g_i'} =
\frac{(a g_i + b g_i') z_i}{a g_i + b g_i'}\cdot 
$$

\begin{theorem}
The affine addition laws for $C$ in $\AA^s$ of bidegree $(m,n)$ form a vector 
space isomorphic to the direct sum of the spaces of addition law projections 
for the coordinate functions $x_1, \dots, x_s$ of bidegree $(m,n)$.
\end{theorem}

\begin{proof}
Every polynomial form $p_i$ in 
$\Gamma(E \times E,\pi_1^*\cO_E(m) \otimes \pi_2^*\cO_E(n))$ 
determines a unique function $f_i = p_i/X_0^m Y_0^n$ in 
$$
k[C] \otimes k[C] = \Gamma(C,\cO_E) \otimes \Gamma(C,\cO_E)
$$ 
and injectivity of $p_i \mapsto f_i$ follows from injectivity of 
$
\Gamma(E,\cO_E(m)) \rightarrow k[C].
$
\qed
\end{proof}

\section{Torsion module structure}
\label{section:torsion-module}


Let $E/k$ be an elliptic curve with finite torsion subgroup $G \subset E(k)$.  
A divisor $D$ is said to be {\it $G$-invariant} if $\tau_P^*D = D$ for all 
$P$ in $G$, where $\tau_P: E \rightarrow E$ is the translation-by-$P$ morphism.  
We hereafter assume that $E/k$ is equipped with a projectively normal 
embedding in $\PP^r$ by $\cL = \cL(D)$, where $D$ is an effective $G$-invariant 
divisor.

\begin{lemma}
Let $\iota: E \rightarrow \PP^r$ be a projectively normal embedding of $E$,
with respect to $\cL$. Let $G$ be a finite torsion subgroup, and suppose that 
$\cL = \cL(D)$ where $D$ is an effective $G$-invariant divisor. Then $G$ 
acts on $E$ by projective linear transformations of $\PP^r$.
\end{lemma}

\begin{proof}
Since $D$ is $G$-invariant, the space $\Gamma(E,\cL)$ has a $k$-linear 
representation by $G$. Since we have a surjective homomorphism 
$
\Gamma(\PP^r,\cO_{\PP^r}(1)) \rightarrow \Gamma(E,\cL), 
$
every linear automorphism of $\Gamma(E,\cL)$ lifts to an automorphism 
of $\Gamma(\PP^r,\cO_{\PP^r}(1))$, hence to a projective linear 
transformation of $\PP^r$.
\qed
\end{proof}

From the action of $\tau_P^*$ on $\Gamma(E,\cL)$, and lifting 
to $\Gamma(\PP^r,\cO_{\PP^r}(1))$, we identify $\tau_P$ with 
a linear polynomial map in $k[X_0,\dots,X_r]^{r+1}$. Let $G_2$ 
be the kernel of the homomorphism 
$
G \times G \times G \rightarrow G
$
defined by $(R,S,T) \mapsto R+S+T$, and let $G_1$ be the subgroup of 
$G_2$ with $T = 0$.  We define the action of $G_2$ (hence of $G_1$) 
on the space of addition laws of bidegree $(m,n)$ by
$(R,S,T) \cdot \fs = \tau_T \circ \fs \circ (\tau_R \times \tau_S)$,
so that
$$
(R,S,T) \cdot \fs(P,Q) = \fs(P+R,Q+S)+T.
$$
Clearly $G_1$ and $G_2$ are isomorphic to $G$ and $G \times G$, 
respectively, with isomorphisms given by $R \mapsto (R,-R,\oO)$ 
and $(R, S) \mapsto (R,S,-R-S)$. 

\begin{lemma}
The group $G_2$ acts linearly on the addition laws of bidegree $(m,n)$. 
\end{lemma}

\begin{proof}
The image $(R,S,T) \cdot \fs$ is the composition of polynomials of 
bidegree $(m,n)$ with linear polynomial maps, which, by the hypothesis 
that $R + S + T = \oO$, determines another addition law.
\qed
\end{proof}

\begin{lemma}
\label{lemma:G-exceptional-divisor-action}
The group $G_2$ acts linearly on the set of divisors of addition laws for 
$E$. In particular the action on the components of addition laws of 
bidegree $(2,2)$ is given by 
$
(R,S,T)^* \Delta_P = \Delta_{P-R+S}.
$
\end{lemma}

\begin{proof}
\label{lemma:divisor-action}
The action on divisors is 
$\mathrm{div}((R,S,T) \cdot \fs) = (\tau_R \times \tau_S)^* \mathrm{div}(\fs)$, 
and the action on $\Delta_P$ follows from 
$$
(\tau_R \times \tau_S)^* \Delta_P = \Delta + (P-R,-S) 
= \Delta + (P-R+S,\oO) = \Delta_{P-R+S}.
$$
Since $T$ determines a linear automorphism of the polynomials of $\fs$, it 
has no bearing on the divisor which they cut out.
\qed
\end{proof}


\begin{theorem}
\label{theorem:eigenvector-divisors}
An addition law $\fs$ is an eigenvector for an element $(R,S,T)$ 
of $G_2$ if and only if the exceptional divisor of $\fs$ is fixed 
by $(R,S,T)$. 
\end{theorem}

The abstract vector spaces of addition laws, as well as the $G_2$-module 
structure are independent of the choice of bases for $\Gamma(E,\cL)$ as 
well as $\Gamma(E \times E,\cM)$.  However, the simplicity of the addition 
laws (as measured, for example, by their sparseness as polynomials) on 
Edwards and Hessian models, is entirely dependent on the choice of the 
sections in $\Gamma(E,\cL)$ and the corresponding coordinate functions of 
the projective embedding, and of the addition laws.   This study grew out 
of the observation that the simplest addition laws arise from the bases 
which arise either as eigenspaces of $G_1$ or which have a permutation 
representation with respect to $G_1$.

For a group $G$ acting linearly on a space of addition laws (for which we 
may consider $G$ of the form $G_1$ or $G_2$ as above), we define an 
addition law $\fs$ to be $G$-complete if $\{ \gamma \fs : \gamma \in G \}$ 
is a geometrically complete set of addition laws (see~\cite{BernsteinKohelLange}).

\section{Addition law constructions}
\label{section:constructions}

\addition{
In this section we apply the above analysis to determine and characterize 
the spaces of addition laws for families with rational torsion subgroups 
or rational torsion points. 
In view of Lemma~\ref{lemma:linear_group_action}, we consider families 
with rational $d$-torsion subgroups for elliptic curve models of 
degree~$d$. 

The complete spaces of addition laws of given bidegree can be determined 
for any effective addition algorithm by interpolating of points 
$((P,Q),\mu(P,Q))$ with monomials of the correct bidegree.  Such an approach 
was suggested by D.~Bernstein and T.~Lange, and a similar interpolation 
algorithm appears in Castryck and Vercauteren~\cite{CastryckVercauteren}.
On a generic model, for which there may exist only finitely many rational 
points, we interpolate points in the formal neighborhood of $\oO$ or 
the rational torsion points.
Hisil et al.~\cite{Hisil-Faster} use an analogous approach through 
Gr\"obner bases, based on an algorithm of Monagan and 
Pierce~\cite{MonaganPierce}, to systematically search for rational 
expressions for affine addition laws.  Using the automorphisms induced 
by torsion points, the spaces of addition laws can be reduced and 
distinguished eigenspaces computed directly.  
Algorithms for the analysis of addition laws and group actions was 
written in Magma~\cite{Magma} and Sage~\cite{Sage}, to be made 
available in Echidna~\cite{Echidna}.

For known families, particularly Edwards curves, the classification 
in terms of eigenspaces explains the canonical nature of the 
distinguished prescribed addition laws reported in the literature. 
}

\addition{
\subsection{Symmetric elliptic curve models of degree 3}

\noindent{\bf Hessian model.} 
The Hessian model $H_d/k: X^3 + Y^3 + Z^3 = dXYZ$ is well known 
as a universal model (over $k(X(3))$) for elliptic curves with 
full torsion subgroup.  
In Bernstein, Kohel, and Lange~\cite{BernsteinKohelLange}, the 
twisted Hessian curves $H_{(a,d)}/k$:
$$
aX^3 + Y^3 + Z^3 = dXYZ,
$$
are introduced (a descent of scalars to $k(X_0(3))$), and their 
addition laws and completeness properties are studied.  
In characteristic different from 3, in terms of the order 3 subgroup 
$G$ defined by $X = 0$, we can characterize the addition laws terms 
of their $G_1$-module structure~\cite{BernsteinKohelLange}.

\begin{theorem}
The space of addition laws of bidegree $(2,2)$ for the twisted 
Hessian curve is spanned by the three addition laws:
$$
\begin{array}{l@{\;}c@{\;}l}
\fs_0 = (\,
 X_1^2 Y_2 Z_2 - Y_1 Z_1 X_2^2, &
 Z_1^2 X_2 Y_2 - X_1 Y_1 Z_2^2, &
 Y_1^2 X_2 Z_2 - X_1 Z_1 Y_2^2\,),\\
\fs_1 = (\,
 X_1 Y_1 Y_2^2 - Z_1^2 X_2 Z_2, &
 a X_1 Z_1 X_2^2 - Y_1^2 Y_2 Z_2, &
 Y_1 Z_1 Z_2^2 - a X_1^2 X_2 Y_2\,),\\
\fs_2 = (\,
 X_1 Z_1 Z_2^2 - Y_1^2 X_2 Y_2, & 
 Y_1 Z_1 Y_2^2 - a X_1^2 X_2 Z_2, &
 a X_1 Y_1 X_2^2 - Z_1^2 Y_2 Z_2\,).
\end{array}
$$
Each $\fs_i$ is an eigenvector for the action of $G_1$.
\end{theorem}

\noindent{\bf Remark.}
The addition laws are also simultaneous eigenvectors for 
the full subgroup $G_2$.  Over an extension in which $a$ 
is a cube root, the curve attains an independent 
$3$-torsion point, which acts by scaled coordinate 
permutation.  Consequently the addition laws are cyclically 
permuted under this action.
This action on the addition law $(4.21i)$ of Chudnovsky 
and Chudnovsky~\cite{ChudnovskyBrothers}, in retrospect,
is sufficient to produce the above basis.
\vspace{1mm}

Similarly an explicit computation yields the following addition law 
projections of bidegrees $(1,2)$ and $(2,1)$.  

\begin{theorem}
\label{theorem:exotic_addition_laws} 
The twisted Hessian curve admits degree $2$ coordinate projections 
$$
(X:X-T), \quad (Y:Y-T), \mbox{ and } (Z:Z-T), 
$$
where $T = X + Y + Z$, for which there exist addition laws of bidegree $(1,2)$:
$$
\begin{array}{l@{\;}c@{\;}}
(\, X_1 Y_2 Z_2 + Y_1 X_2 Y_2 + Z_1 X_2 Z_2 : & X_1 X_2^2 + Y_1 Z_2^2 + Z_1 Y_2^2\,),\\ 
(\, X_1 X_2 Z_2 + Y_1 Y_2 Z_2 + Z_1 X_2 Y_2 : & X_1 Y_2^2 + Y_1 X_2^2 + Z_1 Z_2^2\,),\\
(\, X_1 X_2 Y_2 + Y_1 X_2 Z_2 + Z_1 Y_2 Z_2 : & X_1 Z_2^2 + Z_1 X_2^2 + Y_1 Y_2^2\,), 
\end{array}
$$
and of bidegree $(2,1)$:
$$
\begin{array}{l@{\;}c@{\;}}
(\, Y_1 Z_1 X_2 + X_1 Y_1 Y_2 + X_1 Z_1 Z_2 : & X_1^2 X_2 + Z_1^2 Y_2 + Y_1^2 Z_2\,),\\
(\, X_1 Z_1 X_2 + Y_1 Z_1 Y_2 + X_1 Y_1 Z_2 : & Y_1^2 X_2 + X_1^2 Y_2 + Z_1^2 Z_2\,),\\
(\, X_1 Y_1 X_2 + X_1 Z_1 Y_2 + Y_1 Z_1 Z_2 : & Z_1^2 X_2 + Y_1^2 Y_2 + X_1^2 Z_2\,).
\end{array}
$$
Each addition laws projection spans the unique one-dimensional space of its bidegree.
\end{theorem}

\noindent{\bf Remark.} 
This provides an example of an addition law projection of the critical 
bidegree in Corollary~\ref{corollary:phi-EulerCharacteristic}, at which 
the Euler-Poincar{\'e} characteristic is zero (see the remark following 
Corollary~\ref{corollary:phi-Dimensions}).  Note that the projections 
$(X:X-T)$ and $(X:T)$ are linearly equivalent, but the former yields a 
simpler expression.
}

\addition{
\subsection{Symmetric elliptic curves of degree 4}

Next we consider degree 4 models of elliptic curves, with parametrized 
$2$-torsion and $4$-torsion subgroups.  In order to be diagonalized with 
respect to the torsion subgroup, we assume that the base field is not of 
characteristic $2$. 
\vspace{1mm}

\noindent{\bf Jacobi model.}
Let $J_{(a,b)}$ be the elliptic curve over a field of characteristic different 
from $2$, given by the quadric intersections in $\PP^3$:
$$
\begin{array}{l@{\,}c@{\,}r}
a X_0^2 + X_1^2 & = & X_2^2, \\
b X_0^2 + X_2^2 & = & X_3^2, \\
c X_0^2 + X_3^2 & = & X_1^2, \\
\end{array}
$$
where $a + b + c = 0$, with identity $\oO = (0:1:1:1)$ and $2$-torsion points 
$$
T_1 = (0:-1:1:1),\  
T_2 = (0:1:-1:1),\  
T_3 = (0:1:1:-1).
$$
The embedding in $\PP^3$ is given by a complete linear system associated to 
any divisor equivalent to the sum of the $2$-torsion points, which in 
canonical form of Lemma~\ref{lemma:elliptic-sheaf-canonical} is $4(\oO)$.

\begin{theorem}
Let $E/k$ be an elliptic curve with projective normal embedding in $\PP^3$ 
such that $\cO_E(1) \isom \cL(4(\oO))$.  If $E(k)[2]$ is isomorphic 
to $(\ZZ/2\ZZ)^2$, then there exists $(a,b)$ in $k^2$ such that $E$ is 
linearly isomorphic to $J_{(a,b)}$.
\end{theorem}

\begin{proof}
The $j$-invariant of the family $J_{(a,b)}$ determines an $S_3$ cover 
$j$-line by $(a:b)$ in $\PP^1$, ramified over $j = 0$ and $j = 12^3$, and 
by construction $(a:b)$ represents a point on the modular curve $X(2)$.
Thus for $j$ different from $0$ and $12^3$, it follows that $J_{(a,b)}$ 
encodes a representative elliptic curve with full $2$-torsion, and its 
quadratic twists, associated to each point on $X(2)$. 

For the exceptional values $j = 0$ and $j = 12^3$, we first suppose that 
$\mathrm{char}(k) \ne 3$, so that $0 \ne 12^3$ (since $\mathrm{char}(k) 
\ne 2$ by hypothesis).  An elliptic curve with $j = 0$ or $j = 12^3$ is 
then isomorphic to $y^2 = x^3 - s^3$ or $y^2 = x^3 - s^2x$, respectively.  
In the former case, by hypothesis on the $2$-torsion, there exists $\omega$ 
in $k$ such that $\omega^2 = - \omega - 1$, and cubic and quartic twists 
do not have full $2$-torsion.  Jacobi models for these curves are, respectively,
$$
\ignore{
\begin{array}{r@{\,}c@{\,}l}
  s X_0^2 + X_1^2 &=& X_2^2,\\
  \omega s X_0^2 + X_2^2 &=& X_3^2,\\
  \omega^2 s X_0^2 + X_3^2 &=& X_1^2,
\end{array}
}
\begin{array}{r@{\,}c@{\,}l}
(2\omega + 1)s X_0^2 + X_1^2 &=& X_2^2,\\
w(2\omega + 1)s X_0^2 + X_2^2 &=& X_3^2,\\
w^2(2\omega + 1)s X_0^2 + X_3^2 &=& X_1^2,
\end{array}
\mbox{ and }
\begin{array}{r@{\,}c@{\,}l}
2s X_0^2 + X_1^2 &=& X_2^2,\\
-s X_0^2 + X_2^2 &=& X_3^2,\\
-s X_0^2 + X_3^2 &=& X_1^2.
\end{array}
$$
In characteristic 3, the latter model describes all twists over $k$ of 
the unique supersingular elliptic curve over $\FF_3$ with $j = 12^3 = 0$ 
and full $2$-torsion.  The linearity of the isomorphisms follows from 
Lemma~\ref{lemma:linear_isomorphism}.
\qed
\end{proof}

\noindent{\bf Example.}
Chudnovsky and Chudnovsky~\cite[Section~4]{ChudnovskyBrothers} define a 
Jacobi quadric intersection
$$
\begin{array}{r}
          x^2 + y^2 = 1,\\
\lambda^2 x^2 + z^2 = 1,
\end{array}
$$
which is an affine model for a curve in this family for $(a,b,c) = 
(1,-\lambda^2,\lambda^2-1)$, with the embedding 
$$
(x,y,z) \longmapsto (x:y:1:z).
$$
This gives an example of a nonsingular affine model, which is $k$-complete 
over any field $k$ in which $-1$ is not a square.

Similarly, the projective normal closure of the Jacobi quartic (see 
Section~\ref{section:affine_models}):
$$
X_2^2 = X_0^2 + 2a X_0 X_3 + X_3^2, \quad 
    X_0 X_3 = X_1^2, 
$$
is isomorphic to the Jacobi model with $(a,b,c) = (-2(a+1),4,2(a-1))$,
by the transformation
$
(X_0:X_1:X_2:X_3) \longmapsto (X_1:X_2:X_0-X_3:X_0+X_3).
$

\begin{theorem}
\label{theorem:jacobi-exceptional-divisors}
The space of addition laws of bidegree $(2,2)$ for $J_{(a,b)}$ is 
spanned by $\{ \fs_i : 0 \le i \le 3 \}$, where 
$$
\begin{array}{r@{\,}l}
\fs_0 = ( 
    & X_0^2 Y_1^2 - X_1^2 Y_0^2, \\
    & X_0 X_1 Y_2 Y_3 - X_2 X_3 Y_0 Y_1, \\
    & X_0 X_2 Y_1 Y_3 - X_1 X_3 Y_0 Y_2, \\
    & X_0 X_3 Y_1 Y_2 - X_1 X_2 Y_0 Y_3 ),\\
\\
\fs_2 = (
    & X_0 X_1 Y_2 Y_3 + X_2 X_3 Y_0 Y_1, \\
    & a c X_0^2 Y_0^2 + X_1^2 Y_1^2, \\
    & a X_0 X_3 Y_0 Y_3 + X_1 X_2 Y_1 Y_2, \\
    & -c X_0 X_2 Y_0 Y_2 + X_1 X_3 Y_1 Y_3 ),
\end{array}
\quad
\begin{array}{r@{\,}l}
\fs_1 = (
    & X_0 X_2 Y_1 Y_3 + X_1 X_3 Y_0 Y_2, \\
    & -a X_0 X_3 Y_0 Y_3 + X_1 X_2 Y_1 Y_2, \\ 
    & a b X_0^2 Y_0^2 + X_2^2 Y_2^2, \\
    & b X_0 X_1 Y_0 Y_1 + X_2 X_3 Y_2 Y_3 ),\\
\\
\fs_3 = (
    & a (X_0 X_3 Y_1 Y_2 + X_1 X_2 Y_0 Y_3), \\
    & a (c X_0 X_2 Y_0 Y_2 + X_1 X_3 Y_1 Y_3), \\
    & a (-b X_0 X_1 Y_0 Y_1 + X_2 X_3 Y_2 Y_3), \\
    & -b X_1^2 Y_1^2 - c X_2^2 Y_2^2 ) \\
\end{array}
$$
and the exceptional divisor of $\fs_i$ is $\delta^*(D_i)$ where $D_i$ 
is defined by $X_i = 0$.
\end{theorem}

\begin{proof}
The dimension of the space of addition laws of bidegree $(2,2)$ is four 
by Corollary~\ref{corollary:Dimensions}. The exceptional divisors are 
of the form $\delta^*(D_i)$ by Theorem~\ref{theorem:DeltaIsomorphism}, 
and the divisors $D_i$ are determined by intersecting with $J_{(a,b)} 
\times \{\oO\}$.
\qed
\end{proof}

\begin{corollary}
\label{corollary:Jacobi-eigenspaces}
The addition laws $\fs_0$, $\fs_1$, $\fs_2$ and $\fs_3$ are common 
eigenvectors for the translations $\tau_{T_i}$ and $[-1]$. 
\end{corollary}

\begin{proof}
Since each of the divisors $D_i$ is fixed by $\tau_{T_i}^*$ and $[-1]^*$,
the addition laws are immediately eigenvectors.
\qed
\end{proof}

There exists a torsion point of order $4$ on $J_{(a,b)}$ if and only if 
a pair $\{a,-c\}$, $\{-a,b\}$, or $\{-b,c\}$ consists of squares (namely the 
$4$-torsion points lie on $X_1 = 0$, $X_2 = 0$, or $X_3 = 0$, respectively). 
Any such point then acts linearly on the space $\Gamma(J_{(a,b)},\cL(4(\oO)))$
by Lemma~\ref{lemma:linear_group_action}.

\begin{corollary}
\label{corollary:Jacobi-G-completeness}
Suppose that $G$ is a cyclic subgroup of order $4$ in $J_{(a,b)}(k)$.  
Then any $\fs$ in $\{\fs_0,\fs_1,\fs_2,\fs_3\}$ is $G_2$-complete where 
$G_2$ is defined with respect to $G$.
\end{corollary}

\begin{proof}
The group $G_2$ commutes with the $2$-torsion subgroup, hence induces a 
permutation on the set of eigenspaces spanned by the $\fs_i$.
In view of Lemma~\ref{lemma:G-exceptional-divisor-action}, the group $G_2$ 
includes an element permuting two pairs of eigenspaces.  Since the exceptional 
divisors of Theorem~\ref{theorem:jacobi-exceptional-divisors} are pairwise 
disjoint, any two of the $s_i$ comprise a geometrically complete set.
\qed
\end{proof}
}

\noindent{\bf Edwards models.}
Let $E_1 = E_{(a,d)}$ be the projective normal closure of the 
twisted Edwards model (see Section~\ref{section:affine_models})
$$
X_0^2 + dX_3^2 = aX_1^2 + X_2^2, \quad X_0 X_3 = X_1 X_2. 
$$
In view of the role of the projective addition laws, we define 
its image in $\PP^1 \times \PP^1$:
$$
E_2 : aX^2W^2 + Y^2Z^2 = Z^2W^2 + dX^2Y^2,
$$
given by 
$$
(X_0:X_1:X_2:X_3) \mapsto ((X:Z),(Y:W)) = ((X_0:X_1),(X_0:X_2)),
$$
which is nonsingular.  It follows that the embedding in $\PP^3$ 
is the image of the Segre embedding
$$
((X:Z),(Y:W)) \longmapsto (XY:XW:ZY:ZW) = (X_0:X_1:X_2:X_3).
$$

Here we describe the interplay between the embedding in $\PP^3$ 
and $\PP^1 \times \PP^1$, exploited in the simple addition laws 
of Hisil~\cite{Hisil-EdwardsRevisited} for models in $\PP^3$, 
and interpret the addition laws and their completeness properties 
in terms of eigenspaces under the $4$-torsion subgroup.  
The addition laws so determined on the curve $E_2$ embedded in 
$\PP^1 \times \PP^1$ are those studied by Bernstein and 
Lange~\cite{BernsteinLange-EdwardsComplete}, who prove their 
completeness properties.  The above theory gives a means of 
explaining the canonical nature of these simple addition laws.

Suppose that $c$ and $e$ are square roots of $a$ and $d$, 
respectively, in the algebraic closure of the base field of $E_2$.  
Then $T_1 = (0:1:0:c)$ and $T_2 = (1:0:e:0)$ are points of order $4$, 
and the translation-by-$T_1$ morphism is
$$
(X_0:X_1:X_2:X_3) \longmapsto (-X_0 : c^{-1}X_2 : -cX_1 : X_3),
$$
and that for translation-by-$T_2$ is:
$$
(X_0:X_1:X_2:X_3) \longmapsto (-e^{-1}X_3 : X_1 : -X_2 : eX_0).
$$
We note that $2T_1 = 2T_2 = (0:0:-1:1)$, 
$$
T_1 + T_2 = (-c : e : 0 : 0) \mbox{ and }
T_1 - T_2 = ( c : e : 0 : 0)
$$
and $E_1[2] = \{ \oO, 2T_i, T_1 \pm T_2 \}$.  
Let $G$ be the torsion subgroup $\langle T_1, T_2 \rangle$, isomorphic 
to $\ZZ/2\ZZ \times \ZZ/4\ZZ$. 
We now state the characterization of the spaces of addition laws 
for the group morphism $E_1 \times E_1 \rightarrow E_2$, in terms of 
bases of distinguished eigenvectors and their exceptional divisors. 
These addition laws, as well as the characterization of exceptional 
divisors, can be deduced from the addition laws for $E_2 \times E_2 
\rightarrow E_2$ of Bernstein and Lange~\cite{BernsteinLange-EdwardsComplete}, 
by factoring through the Segre embedding (see note below 
Corollary~\ref{corollary:EdwardsE2}).

\begin{theorem}
\label{theorem:edwards-exceptional-divisors}
The space of addition laws for $E_1 \times E_1 \rightarrow E_2$ of 
bidegree $(1,1)$ is spanned by $\{ (\fs_i,\ft_j) : 0 \le i, j \le 1 \}$, 
where 
$$
\begin{array}{l}
\fs_0 = (X_0 Y_3 + X_3 Y_0,\ a X_1 Y_1 + X_2 Y_2), \\
\fs_1 = (X_1 Y_2 + X_2 Y_1,\ d X_0 Y_0 + X_3 Y_3),
\end{array}
$$
with respective exceptional divisors
$
\Delta_{T_1} + \Delta_{-T_1} \mbox{ and } \Delta_{T_2} + \Delta_{-T_2},
$
and
$$
\begin{array}{l}
\ft_0 = (X_0 Y_3 - X_3 Y_0,\ X_1 Y_2 - X_2 Y_1), \\
\ft_1 = (a X_1 Y_1 - X_2 Y_2,\ d X_0 Y_0 - X_3 Y_3),
\end{array}
$$
with respective exceptional divisors
$
\Delta_\oO + \Delta_{2T_i} \mbox{ and } \Delta_{T_1+T_2} + \Delta_{T_1-T_2}.
$
\end{theorem}

\begin{proof}
The correctness of the addition laws is verified by explicit substitution.  
The dimension of each of the addition law projections is $2$, in 
accordance with Corollary~\ref{corollary:phi-Dimensions} and 
the degrees of the projections of $E_2$ to $\PP^1$. 
Thus the two sets $\{\fs_0, \fs_1\}$ and $\{\ft_0, \ft_1\}$ are bases 
for the spaces of addition law projections. Correctness of the exceptional 
divisors can be verified by intersection with $E \times \{\oO\}$.
\qed
\end{proof}

Let $G_1$ and $G_2$ be the subgroups defined in the previous section,
with respect to the group $G = \langle T_1, T_2 \rangle$.  
The group $G_1$ has a well-defined action on the two spaces spanned 
by $\{\fs_0,\fs_1\}$ and $\{\ft_0,\ft_1\}$, while the action of 
$G_2$ only becomes well defined on the span of tuples $\{(\fs_i,\ft_j)\}$.

\begin{corollary}
The sets $\{ \fs_0, \fs_1 \}$ and $\{ \ft_0, \ft_1 \}$ are stabilized by $G_1$ 
and pointwise fixed by the subgroup $\langle (2T_i,2T_i,\oO) \rangle$.
Moreover each of $k\fs_j$ and $k\ft_j$ are eigenspaces for the action of $G_1$. 
The action of $G_2$ stabilizes the sets of 
pairs 
$\{ (k\fs_0,k\ft_0), (k\fs_1,k\ft_1) \}$ and 
$\{ (k\fs_0,k\ft_1), (k\fs_1,k\ft_0) \}$, and acts transitively on their product.
\end{corollary}

\begin{proof}
By Theorem~\ref{theorem:eigenvector-divisors}, the eigenvectors 
are characterized by the action on the exceptional divisors.  
By Lemma~\ref{lemma:divisor-action} and the form of the exceptional 
divisors in Theorem~\ref{theorem:edwards-exceptional-divisors}, we see 
that the exceptional divisors are stabilized by $(T_i,-T_i,\oO)$ and 
hence $\fs_0$, $\fs_1$, $\ft_0$ and $\ft_1$ are eigenvectors.
By explicit substitution we find eigenvalues $(-1,1,-1,1)$ for $T_1$ 
and eigenvalues $(1,-1,-1,1)$ for $T_2$.  Hence each of the spaces 
spanned by $\{ \fs_0, \fs_1 \}$ and $\{ \ft_0, \ft_1 \}$ decomposes 
into one-dimensional eigenspaces.  The action on eigenspace pairs 
follows similarly from the action on exceptional divisors. 
\qed
\end{proof}

\begin{theorem}
The addition law projection $\fs_0$, $\fs_1$, or $\ft_1$ is $k$-complete 
if and only if $a$, $d$, or $ad$ is a nonsquare, respectively.  
In particular, over a finite field, either zero or two of $\fs_0$, $\fs_1$ 
and $\ft_1$ are $k$-complete.
\end{theorem}

\begin{proof}
The sets $\{T_1,-T_1\}$, $\{T_2,-T_2\}$ and $\{T_1+T_2,T_1-T_2\}$ are 
Galois orbits of non-$k$-rational points when $a$, $d$, or $ad$ is a 
nonsquare, respectively, in which case the respective divisor  
$\Delta_{T_1} + \Delta_{-T_1}$, 
$\Delta_{T_2} + \Delta_{-T_2}$ or 
$\Delta_{T_1+T_2} + \Delta_{T_1-T_2}$, is irreducible over $k$ and 
hence has no rational point.  Over a finite field, either zero or 
two of $a$, $d$, and $ad$ are nonsquares. 
\qed
\end{proof}
Let $\varphi: E_2 \times E_2 \rightarrow E_1$ be the restriction 
of the Segre embedding $\PP^1 \times \PP^1 \rightarrow \PP^3$, and 
identify $\varphi$ with the polynomial map 
$$
((X,Z),(Y,W)) \mapsto (XY,XW,ZY,ZW). 
$$
As a consequence of the above theorem, the four-dimensional 
space of addition laws for $E_1$ is obtained in factored form 
as the pairwise combination of these pairs of addition laws, 
under the Segre embedding in $\PP^3$.

\begin{corollary}
\label{corollary:EdwardsE1}
The space of addition laws of bidegree $(2,2)$ for 
$$
\mu: E_1 \times E_1 \longrightarrow E_1
$$
is spanned by $\{ \varphi(\fs_i,\ft_j) : 0 \le i, j \le 1 \}$.
\end{corollary}

Similarly, we obtain a factored form $E_2 \times E_2 \rightarrow 
E_1 \times E_1 \rightarrow E_2$ for the addition laws on $E_2$. 

\begin{corollary}
\label{corollary:EdwardsE2}
The space of addition laws of multidegree $((1,1),(1,1))$ for 
$$
\mu: E_2 \times E_2 \longrightarrow E_2
$$
is spanned by $\{ (\fs_i \circ \varphi \times \varphi,\ft_j \circ \varphi \times \varphi) : 0 \le i, j \le 1 \}$.
\end{corollary}

In expanded form Corollary~\ref{corollary:EdwardsE1} gives the addition laws:
$$
\begin{array}{l}
\multicolumn{1}{l}{\hspace{-2mm}
\varphi(\fs_0,\ft_0) = 
  \big(
  (X_0 Y_3 + X_3 Y_0)(X_0 Y_3 - X_3 Y_0),\,
  (X_0 Y_3 + X_3 Y_0)(X_1 Y_2 - Y_1 X_2),}\\
\multicolumn{1}{r}{
  (a X_1 Y_1 + X_2 Y_2)(X_0 Y_3 - X_3 Y_0),\, 
  (a X_1 Y_1 + X_2 Y_2)(X_1 Y_2 - Y_1 X_2) 
  \big),} \\[1mm]
\multicolumn{1}{l}{\hspace{-2mm}
\varphi(\fs_0,\ft_1) = 
  \big(
  (X_0 Y_3 + X_3 Y_0)(a X_1 Y_1 - X_2 Y_2),\, 
  (X_0 Y_3 + X_3 Y_0)(a X_1 Y_1 - X_2 Y_2),}\\
\multicolumn{1}{r}{
  (a X_1 Y_1 + X_2 Y_2)(d X_0 Y_0 - X_3 Y_3),\,
  (a X_1 Y_1 + X_2 Y_2)(d X_0 Y_0 - X_3 Y_3) 
  \big),} \\[1mm]
\multicolumn{1}{l}{\hspace{-2mm}
\varphi(\fs_1,\ft_0) = 
  \big(
  (X_1 Y_2 + X_2 Y_1) (X_0 Y_3 - X_3 Y_0),\, 
  (X_1 Y_2 + X_2 Y_1) (X_1 Y_2 - Y_1 X_2),}\\ 
\multicolumn{1}{r}{
  (d X_0 Y_0 + X_3 Y_3) (X_0 Y_3 - X_3 Y_0),\,  
  (d X_0 Y_0 + X_3 Y_3) (X_1 Y_2 - Y_1 X_2)
  \big),} \\[1mm]
\multicolumn{1}{l}{\hspace{-2mm}
\varphi(\fs_1,\ft_1) = 
  \big(
  (X_1 Y_2 + X_2 Y_1) (a X_1 Y_1 - X_2 Y_2),\, 
  (X_1 Y_2 + X_2 Y_1) (d X_0 Y_0 - X_3 Y_3),}\\ 
\multicolumn{1}{r}{
  (d X_0 Y_0 + X_3 Y_3) (a X_1 Y_1 - X_2 Y_2),\, 
  (d X_0 Y_0 + X_3 Y_3) (d X_0 Y_0 - X_3 Y_3)
  \big).}
\end{array}
$$
The forms $\varphi(\fs_1,\ft_1)$ and $\varphi(\fs_0,\ft_0)$, with given 
factorization, appear as equations $(5)$ and $(6)$, respectively, in Hisil 
et al.~\cite{Hisil-EdwardsRevisited}.
Similarly, in expanded form Corollary~\ref{corollary:EdwardsE2} gives the 
addition law projections of Bernstein and Lange~\cite{BernsteinLange-EdwardsComplete}:
$$
\begin{array}{l}
\fs_0 \circ \varphi \times \varphi 
  = (X_1 Y_1 Z_2 W_2 + Z_1 W_1 X_2 Y_2,\ a X_1 W_1 X_2 W_2 + Z_1 W_1 Z_2 W_2),\\
\fs_1 \circ \varphi \times \varphi 
  = (X_1 W_1 Z_2 Y_2 + Z_1 Y_1 X_2 W_1,\ d X_1 Y_1 X_2 Y_2 + Z_1 W_1 Z_2 W_2),\\
\ft_0 \circ \varphi \times \varphi 
  = (X_1 Y_1 Z_2 W_2 - Z_1 W_1 X_2 Y_2,\ X_1 W_1 Z_2 Y_2 - X_1 W_1 Z_2 Y_2), \\
\ft_1 \circ \varphi \times \varphi 
  = (a X_1 W_1 X_2 W_2 - Z_1 Y_1 Z_2 Y_2,\ d X_1 Y_1 X_2 Y_2 - Z_1 W_1 Z_2 W_2).
\end{array}
$$
The set of exceptional divisors of these addition laws, described in Bernstein 
and Lange~\cite[Section~8]{BernsteinLange-EdwardsComplete}, is equivalent to that 
of Theorem~\ref{theorem:edwards-exceptional-divisors}, since the Segre embedding 
is globally defined by a single polynomial map with trivial exceptional divisor. 
\vspace{2mm}

\addition{
\noindent{\bf Canonical curve of level 4.}
In light of the simple structure of the twisted Hessian curve, we define a 
canonical model $C/k$ of level $n$ to be an elliptic curve with subgroup 
scheme $G \isom \mu_n$, embedded in $\PP^{r}$ for $r = n-1$. Moreover we 
assume that there exists $T$ in $G(k(\zeta))$, for an $n$-th root of unity 
$\zeta$ in $\bar{k}$, such that  
$$
\tau_T(X_0:X_1:\cdots:X_r) \longmapsto
  (X_0:\zeta X_1:\cdots: \zeta^r X_r).
$$
Moreover, there exists $S$ in $C(\bar{k})$ such that 
$\langle S, T \rangle = C[n]$ and for some $a_0,\dots,a_r$ in $\bar{k}$,
$$
\tau_S(X_0:X_1:\cdots:X_r) \longmapsto
  (a_1 X_1: \cdots: a_r X_r: a_0 X_0).
$$
This generalizes the Hessian model and the diagonalized Edwards model 
(of the $-1$ twist).

The Edwards curve, with $a = 1$, 
$$
\begin{array}{c}
X_0^2 + d X_3^2 = X_1^2 + X_2^2, \\
X_0 X_3 = X_1 X_2,
\end{array}
$$
has $4$-torsion point $S = (1:1:0:0)$ such that $\tau_S$ is:
$$
\tau_S(X_0:X_1:X_2:X_3) = (X_0:X_2:-X_1:-X_3),
$$
defined by the matrix 
$$
\left(
\begin{array}{@{\,}rr@{\,}r@{\,}r@{\,}}
  1 &   &   &   \\
    & 0 &-1 &   \\
    & 1 & 0 &   \\
    &   &   &-1
\end{array}\right)
$$
which we wish to diagonalize.  First we twist by $a = -1$ so 
that the diagonalization descends, and from the twisted Edwards 
curve, with $a = -1$, 
$$
X_0^2 - d X_3^2 = -(X_1 - X_2)(X_1 + X_2), \
X_0 X_3 = X_1 X_2,
$$
we find the {\it canonical curve $C$ of level 4}: 
$$
\begin{array}{c}
X_0^2 - d X_2^2 = X_1 X_3,\\
X_1^2 - X_3^2 = 4 X_0 X_2,
\end{array}
$$
via the isomorphism
$$
(X_0:X_1:X_2:X_3) \longmapsto (X_0:X_1+X_2:X_3:-X_1+X_2).
$$
This curve has identity $(1:1:0:1)$ and the point $(i : 1 : 0 : 0)$
on $E$ maps to $(1 : i : 0 : -i)$ on $C$, which acts by
$$
(X_0:X_1:X_2:X_3) \longmapsto (X_0:iX_1:-X_2:-iX_3).
$$

\begin{theorem}
The space of addition laws of bidegree $(2,2)$ for the canonical 
model of level $4$ is spanned by : 
$$
\begin{array}{r@{\,}l}
\fs_0 = ( 
    & -(X_1^2 Y_3^2 - X_3^2 Y_1^2)/4, \\
    & X_0 X_3 Y_1 Y_2 - X_1 X_2 Y_0 Y_3, \\
    & X_0^2 Y_2^2 - X_2^2 Y_0^2, \\
    & X_0 X_1 Y_2 Y_3 - X_2 X_3 Y_0 Y_1), \\ \\
\fs_2 = (
    & X_0^2 Y_0^2 - d^2 X_2^2 Y_2^2, \\
    & X_0 X_1 Y_0 Y_1 - d X_2 X_3 Y_2 Y_3, \\
    & (X_1^2 Y_1^2 - X_3^2 Y_3^2)/4, \\
    & X_0 X_3 Y_0 Y_3 - d X_1 X_2 Y_1 Y_2), \\
\end{array}
\begin{array}{r@{\,}l}
\fs_1 = (
    & X_0 X_1 Y_0 Y_3 + d X_2 X_3 Y_1 Y_2, \\
    & 4 d X_0 X_2 Y_2^2 + X_1^2 Y_1 Y_3, \\
    & X_0 X_3 Y_2 Y_3 + X_1 X_2 Y_0 Y_1, \\
    & X_1 X_3 Y_3^2 - 4 d X_2^2 Y_0 Y_2), \\ \\
\fs_3 = (
    & X_0 X_3 Y_0 Y_1 + d X_1 X_2 Y_2 Y_3, \\
    & X_1 X_3 Y_1^2 + 4 d X_2^2 Y_0 Y_2, \\
    & X_0 X_1 Y_1 Y_2 + X_2 X_3 Y_0 Y_3, \\
    & -4 d X_0 X_2 Y_2^2 + X_3^2 Y_1 Y_3). \\
\end{array}
$$
\end{theorem}
}

\subsection{Symmetric elliptic curve models of degree 5}

\addition{
In analogy with the Hessian model and canonical model of level 4,
we describe the construction of a canonical model of level 5, 
which we call pentagonal elliptic curves.  As with the canonical 
models of levels 3 and 4, the addition laws have simple expressions 
in terms of differences of monomials.
}
\vspace{2mm}

\noindent{\bf Pentagonal elliptic curves.}
We describe a model for elliptic curves over the function field $k(t)$ 
of $X_1(5)$.  Let $E/k(t)$ be the elliptic curve in $\PP^4$ defined by 
$$
\begin{array}{c}
t U_0^2 + U_2 U_3 - U_1 U_4 =  
t U_0 U_1 + U_2 U_4 - U_3^2 = 
U_1^2 + U_0 U_2 - U_3 U_4 = 0\\
U_1 U_2 + U_0 U_3 - U_4^2 = 
U_2^2 - U_1 U_3 + t U_0 U_4 = 0,
\end{array}
$$
with base point $\oO = (0:1:1:1:1)$. This model is derived from an 
input Weierstrass model $E$ over $k(t)$ by computing the Riemann-Roch space 
$\Gamma(E,\cL(G))$ where $G = \langle T \rangle$ is a cyclic subgroup 
of order $5$, considered as a divisor on $E$.  The coordinate functions 
$U_i$ are determined by a choice of basis of eigenfunctions for the 
translation-by-$T$ map. For a 5-th root of unity $\zeta$, the image 
of $T$ is $(0:\zeta:\zeta^2:-\zeta^3:-\zeta^4)$ and translation-by-$T$ 
induces:
$$
(U_0:U_1:U_2:U_3:U_4) \longmapsto 
(U_0:\zeta U_1:\zeta^2 U_2:\zeta^3 U_3:\zeta^4 U_4).
$$
We note that the projection to $(U_0:U_1:U_4)$ yields a plane model 
$$
U_1^5 + U_4^5 - (t - 3) U_1^2 U_4^2 U_0 + (2 t - 1) U_1 U_4 U_0^3 - t U_0^5,
$$
but that being singular the dimension formulas fail to apply.
Indeed there are no bidegree $(2,2)$ addition laws for this 
planar model. 

\begin{theorem}
The space of addition laws of bidegree $(2,2)$ on $E$ is of dimension $5$ 
and decomposes over $k(t)$ into eigenspaces for the action of $G_1$. 
The eigenspace for $1$ is given by the polynomial maps:
$$
\begin{array}{l}
(\, U_0^2 V_1 V_4 - U_1 U_4 V_0^2  
  = (U_1 U_4 V_2 V_3 - U_2 U_3 V_1 V_4)/t 
  = - U_2 U_3 V_0^2 + U_0^2 V_2 V_3 : 
\\
\ \ U_0 U_1 V_2 V_4 - U_2 U_4 V_0 V_1 
  = (-U_2 U_4 V_3^2 + U_3^2 V_2 V_4)/t 
  = U_0 U_1 V_3^2 - U_3^2 V_0 V_1 : 
\\
\ \ U_0 U_2 V_3 V_4 - U_3 U_4 V_0 V_2 
  = U_0 U_2 V_1^2 - U_1^2 V_2 V_0 
  = -U_1^2 V_3 V_4 + U_3 U_4 V_1^2 : 
\\
\ \ U_0 U_3 V_1 V_2 - U_1 U_2 V_0 V_3 
  = U_0 U_3 V_4^2 - U_4^2 V_0 V_3 
  = - U_1 U_2 V_4^2 + U_4^2 V_1 V_2 : 
\\
\ \ U_0 U_4 V_1 V_3 - U_1 U_3 V_0 V_4 
  = U_0 U_4 V_2^2 - U_2^2 V_0 V_4
  = (U_1 U_3 V_2^2 - U_2^2 V_1 V_3)/t).
\end{array}
$$
\end{theorem}

\ignore{
$$
\begin{array}{l}
(\, U_0 U_3 V_1^2 + U_4^2 V_0 V_2 : \\
\ \ U_1 U_3 V_1 V_2 + t U_0 U_4 V_0 V_3 : \\
\ \ U_2 U_3 V_1 V_3 - t U_1 U_4 V_0 V_4 : \\
\ \ -t U_2 U_4 V_0^2 + U_3^2 V_1 V_4 : \\
\ \ -t U_1^2 V_0 V_1 + U_3 U_4 V_3^2\,), 
\end{array}
$$
$$
\begin{array}{l}
(\, (U_2 U_4 V_2^2 - U_3^2 V_1 V_3)/t : \\
\ \ U_0 U_2 V_2 V_3 - U_3 U_4 V_1 V_4 : \\
\ \ t U_0 U_3 V_0 V_1 - U_1 U_2 V_2 V_4 : \\ 
\ \ -(U_1 U_3 V_1^2 + U_2^2 V_0 V_2) : \\ 
\ \ -(t U_0^2 V_1 V_2 + U_2 U_3 V_4^2)\,), 
\end{array}
$$
$$
\begin{array}{l}
(\, U_1 U_3 V_3^2 - U_2^2 V_2 V_4)/t : \\
\ \ U_1 U_4 V_3 V_4 - U_2 U_3 V_0 V_2 : \\ 
\ \ U_2 U_4 V_1 V_2 - t U_0 U_1 V_0 V_3 : \\
\ \ U_0 U_2 V_2^2 + U_1^2 V_1 V_3 : \\
\ \ t U_1 U_2 V_0^2 + U_4^2 V_2 V_3\,), 
\end{array}
$$
$$
\begin{array}{l}
(\, U_1^2 V_0 V_3 + U_0 U_2 V_4^2 : \\
\ \ U_1 U_2 V_1 V_3 + t U_0 U_3 V_0 V_4 : \\ 
\ \ U_1 U_3 V_2 V_3 - t U_0 U_4 V_1 V_4 : \\ 
\ \ U_1 U_4 V_3^2 - t U_0^2 V_2 V_4 : \\
\ \ -t U_0 U_1 V_1^2 + U_3^2 V_3 V_4\,). 
\end{array}
$$
}

\noindent{\bf Remark.}
The function $t$ can be identified with a modular function generating the 
function field of $X_1(5)$.  The modular curve $X(5)$ is also of genus $0$, 
and there exists a modular function $e$ satisfying $t = e^5$ which generates 
the function field of $X(5)$.
Over this extension the $5$-torsion point $S = (1:e:-e^2:e^3:0)$, and the 
translation-by-$S$ morphism is:
$$
(U_0 : U_1 : U_2 : U_3 : U_4 ) \longmapsto 
(-U_4:e^4 U_0 : e^3 U_1 : -e^2 U_2 : e U_3).
$$
The remaining eigenspaces of addition laws are permuted by the action induced 
by the subgroup $G = \langle S \rangle$.  In particular, since the action is 
a scaled monomial permutation, the remaining eigenspaces are also described by 
binomial biquadratic polynomials.
\vspace{4mm}

\noindent{\bf Acknowledgement.}
The author thanks Dan Bernstein and Tanja Lange for helpful discussions and 
motivation to undertake this study. Moreover this work benefited from 
discussion with Christophe Ritzenthaler, pointers from Marc Hindry, and 
comments from Steven Galbraith. Finally, the author thanks the various 
anonymous referees for careful reading and comments leading to an improvement 
of the article.

\end{document}